\documentclass[letterpaper, 12pt]{amsart}
\usepackage[left=0.8 in, right=0.8 in, top=1 in, bottom=1 in]{geometry}

\usepackage{amsmath,amssymb,amsthm,amscd,enumerate,setspace}

\usepackage[pagebackref]{hyperref}
\hypersetup{colorlinks=true, linkcolor=red, citecolor=blue}

\usepackage[all]{xy}
\usepackage{graphicx}

\usepackage{xcolor}
\usepackage{mdframed}
\newmdenv[backgroundcolor=yellow]{shaded}

\input xypic
\xyoption{all}

\long\def\symbolfootnote[#1]#2{\begingroup%
\def\thefootnote{\fnsymbol{footnote}}\footnote[#1]{#2}\endgroup}

\newtheorem{Theorem}{Theorem}[section]

\newtheorem{Corollary}[Theorem]{Corollary}
\newtheorem{Proposition}[Theorem]{Proposition}

\newtheorem{Conjecture}[Theorem]{Conjecture}

\newtheorem{Definition}[Theorem]{Definition}
\newtheorem{Question}[Theorem]{Question}

\newtheorem{Sticky Points}[Theorem]{Sticky Points}

\def\Ass{\mbox{\rm Ass}}

\def\cl{\overline}

\def\deg{\mbox{\rm deg}}
\def\Deg{\mbox{\rm Deg}}
\def\depth{\mbox{\rm depth}}
\def\ds{\displaystyle}

\def\edeg{\mbox{\rm embdeg}}
\def\edim{\mbox{\rm embdim}}

\def\Ext{\mbox{\rm Ext}}

\def\G{\mathcal{G}}

\def\h{\mbox{\rm ht}}
\def\hdeg{\mbox{\rm hdeg}}

\def\Jac{\mbox{\rm Jac}}

\def\lar{\longrightarrow}
\def\l{{\lambda}}

\def\rar{\rightarrow}

\def\Tor{\mbox{\rm Tor}}

\def\tratto{\mbox{\rule{3mm}{.2mm}$\;\!$}}

\def\NN{\mathbb{N}}
\def\QQ{\mathbb{Q}}

\def\ZZ{\mathbb{Z}}

\def\bfx{\mathbf{x}}

\def\clA{\mathcal{A}}

\def\clE{\mathcal{E}}
\def\clF{\mathcal{F}}

\def\clM{\mathcal{M}}
\def\clN{\mathcal{N}}

\def\clR{\mathcal{R}}

\def\fkm{\mathfrak{m}}
\def\fkn{\mathfrak{n}}
\def\fkp{\mathfrak{p}}
\def\fkq{\mathfrak{q}}

\def\fkP{\mathfrak{P}}

\def\rme{\mathrm{e}}
\def\rmh{\mathrm{h}}

\def\rmH{\mathrm{H}}
\def\rmP{\mathrm{P}}
\def\rmS{\mathrm{S}}

\def\gla{\alpha}

\begin{document}

\title{\sc Hilbert Coefficients and Sally Modules: A Survey of Vasconcelos' Contributions}

\thanks{
AMS 2020 {\em Mathematics Subject Classification}.
Primary 13-02;  Secondary 13D40, 13H10, 13H15.}

\author{Jooyoun Hong}
\address{Department of Mathematics, Southern Connecticut State
University, 501 Crescent Street, New Haven, CT 06515-1533.}
\email{hongj2@southernct.edu}

\author{Susan Morey} \address{Department of Mathematics, Texas State University, 601 University Drive, San Marcos, TX 78666.}
 \email{morey@txstate.edu}

\maketitle

{\em \small This paper is written in loving memory of Wolmer Vasconcelos, a brilliant mathematician and most caring advisor whose contributions to the field of Commutative Algebra cannot be overstated.}

\begin{abstract} 
This paper surveys and summarizes Wolmer Vasconcelos' results surrounding multiplicities, Hilbert coefficients, and their extensions. We particularly focus on Vasconcelos' results regarding multiplicities and Chern coefficients, and other invariants which they bound. The Sally module is an important instrument introduced by Vasconcelos for this study, which naturally relates Hilbert coefficients to reduction numbers.
\end{abstract}

\section{Introduction}

Throughout his career, Wolmer Vasconcelos had an abiding interest in exploring how invariants of ideals and modules are interrelated. In particular, he was interested in the information that could be gleaned from information stemming from Hilbert polynomials.  The coefficients of the Hilbert polynomial are called the {\it Hilbert coefficients} of the ideal or module (satisfying certain conditions) and are denoted by $\rme_i$ (see Section~\ref{background}). The leading coefficient $\rme_{0}$ is called the {\em multiplicity}.  The constant ${\rme_{1}}$ is called the {\em Chern coefficient} and  is particularly useful in understanding  the structure of rings and modules. In this paper, we provide an overview of Vasconcelos' work related to Hilbert coefficients with a particular focus on the multiplicity and the Chern coefficient, and their interplay.

\medskip

Vasconcelos' interest in multiplicity was sparked by the question of how the multiplicity could be used to bound other invariants such as the number of generators of the ideal. He was also interested in the relationships between the multiplicity and the other Hilbert coefficients and in what information could be obtained from vanishing of these invariants. It is known that the Chern coefficient $\rme_{1}(Q, M)$ of a Cohen-Macaulay module $M$ relative to a parameter ideal  $Q$ vanishes. Vasconcelos conjectured that the converse holds true for an unmixed module (Conjecture~\ref{Ch1-C3-1}). Vasconcelos' exploration of such vanishings led to the definition of a {\it Vasconcelos  module}, which is a finitely generated module $M$ such that $\rme_{1}(Q, M)=0$ for some parameter ideal $Q$ for $M$ (Definition~\ref{Vasconcelos}).  Generally speaking, a module being Vasconcelos can be viewed as a generalization of Cohen-Macaulayness. Once the Vanishing Conjecture was settled (see Theorem~\ref{Ch5-3-1}), natural progressions to other possible cases of Chern coefficients were made by Vasconcelos and other researchers. For example, if a Chern coefficient of a unmixed module is constant (but not necessarily zero) and  independent of the choice of parameter ideal, then the module is Buchsbaum (Theorem~\ref{Ch5-5-4}).

\medskip

Another application of  Chern coefficients of interest to Vasconcelos was their ability to capture the computational complexity of  the integral closure. To this end, he determined a chain of modules on which $\rme_{1}$ is non-decreasing, allowing the use of $\rme_{1}$ to obtain numerical criteria for normality (See Theorem~\ref{closure chain}).

\medskip

One of his greatest contributions to the field of Hilbert coefficients was his introduction of {\it Sally modules} in \cite{V94} to give a quick access to Judith Sally's results connecting the depth properties of associated graded rings to the multiplicity. 
The impact of this heavily cited work continues to this day as Sally modules can be used  to understand higher Hilbert coefficients as well. Furthermore, the multiplicity of the Sally module is an important instrument used to find bounds on reduction numbers of ideals. The beautiful interweaving of invariants and reductions led to Vasconcelos' love of Sally modules that can be seen through their frequent appearance in his subsequent work and that of some of his doctoral students.

\medskip

Vasconcelos embarked on an ambitious project to extend the notion of Hilbert coefficients (which is restricted to $\fkm$-primary ideals) to more general classes of ideals and modules. Vasconcelos explored generalizations from parameter ideals to ideals generated by $d$-sequences. He also studied partial Euler characteristic functions that could be applied to modules that were not unmixed, yielding results similar to those using Chern coefficients in the unmixed case. When the ideal is not $\fkm$-primary or the focus is a parameter ideal generated by a partial system of parameters,  Vasconcelos studied non-traditional Hilbert coefficients, such as the $j$-coefficients. Using the associated graded ring and taking the $0^{th}$ local cohomology proved to be  useful when a full system of parameters was not available.  Later, Vasconcelos defined additional degree functions that could be applied to more general ideals and modules that generalized Hilbert coefficients. His ultimate goal was to determine one unified degree function that would apply to a very broad class of ideals and modules and when restricted to special classes, such as $\fkm$-primary ideals, would reduce to the known Hilbert coefficients.

\medskip

We now provide an overview of the structure of the paper. In Section~\ref{background}, we provide definitions and background material, as well as fix uniform notation for the paper. In Section~\ref{early}, we present some of Vasconcelos' early results on Hilbert coefficients, the Sally module and the relationship between $\rme_{0}(I)$ and $\rme_1(I)$ (See Theorem~\ref{one-dim}). It is particularly interesting to note that in Theorem~\ref{Sally-series}, which fine tunes the result on Hilbert series given in Theorem~\ref{RVV01-2-2}, that the multiplicity of the Sally module appears as a
degree of the numerator of a quotient that bounds the Hilbert series. 

\medskip

Section~\ref{mult} focuses on the multiplicity of an ideal or module. A number of Vasconcelos' results showing how the multiplicity can be used to bound the number of generators of an ideal and the reduction number of an ideal are presented. 
In Section~\ref{normal}, the connections of Chern coefficients to the complexity of computing the integral closure of an ideal are presented. The Brian\c{c}on-Skoda number is used in several of the theorems. The latter part of Section~\ref{normal} provides more general results that apply to a broad class of filtrations that includes $I$-adic filtration $\{I^n\}$ and normal filtration $\{\overline{I^n}\}$. These results show how a generalized version of Hilbert coefficients for filtrations can provide useful bounds (See Theorem~\ref{filtration}). Other invariants, such as embedding dimension and embedding degree, are also important measures of complexity of computations, as shown in Theorem~\ref{embedding}.

\medskip

Section~\ref{Vrings} focuses on Vasconcelos' Vanishing Conjecture~\ref{Ch1-C3-1} that led to the study of Vasconcelos modules. We summarize results showing when this conjecture is known to hold true before it was settled affirmatively in Theorem~\ref{Ch5-3-1}. In Section~\ref{Chern}, we discuss a natural extension of the vanishing of Chern coefficients to other possible sets of Chern coefficients. We also present Vasconcelos' Uniformity Conjecture~\ref{Ch1-Conj3} and Positivity Conjecture~\ref{Conj2}, and results addressing when these conjectures hold true.

\medskip

Section~\ref{Euler} contains results motivated by the question of whether there are other invariants, in addition to Chern coefficients, that can detect Cohen-Macaulayness and other important properties of rings and modules. It turns out a  partial Euler characteristic $\chi_1$ is has properties similar to those of Chern coefficients $\rme_{1}$ but can be used in the more general case of a module that need not be unmixed. A $d$-sequence plays a role in providing a direct relation between $\rme_1$ and $\chi_1$ as shown in Theorems~\ref{Ch7-3-7} and~\ref{Ch7-4-2}. To handle ideals that need not be $\fkm$-primary or partial systems of parameters, Vasconcelos and his coauthors used the $j$-coefficients and particularly $j_{1}$ to detect various properties of modules.

\medskip 

In Section~\ref{Variations}, we discuss how  $\rme_0(J)$ and $\rme_1(J)$ change when $J$ is enlarged. 
The variation $\rme_1(I) - \rme_1(Q)$, where $Q$ is a minimal reduction of $I$, can be used to describe the multiplicity of the Sally module and an upper bound for a reduction number in a non-Cohen-Macaulay ring (Theorem~\ref{min red})
Section~\ref{Sally} returns to the study of Sally modules and their multiplicities. These are related to the study of the special fiber cone and its multiplicity. Proposition~\ref{Sally mult prop} focuses on bounding the Sally multiplicity and Proposition~\ref{fiber cone mult} focuses on bounding the multiplicity of the special fiber, also called the fiber cone, of $I$. We conclude the paper by briefly introducing some of Vasconcelos' extended degree functions in Section~\ref{epilogue}.

\bigskip

\section{Preliminaries}\label{background}

In this section we collect basic definitions and results that will be used throughout the paper. We will also fix a set of notations that will be used. For additional background material, see \cite{BH}.

\medskip

Let $\clR=\bigoplus_{n \geq 0} R_{n}$ be a Noetherian nonnegatively graded ring with $R_{0}$ Artinian local and let
$\clM=\bigoplus_{i \in \ZZ} M_i$ be a finitely generated graded $\clR$--module. Then each $M_i$ has finite length as $R_0$--module. Let $\NN_{0}$ be the set of all nonnegative integers.  The {\em Hilbert function} of $\clM$ is the  function $\rmH_{\clM}: \ZZ \rar \NN_{0}$ defined by $\rmH_{\clM}(n)=\l_{R_0}(M_n)$, where $\l(\tratto)$ denotes the length of a module. The {\em Hilbert series} of $\clM$ is  the Laurent series $\rmh_{\clM}(t)=\sum_{n \in \ZZ} \rmH_{M}(n)t^n$. In particular, if $\clR=R_{0}[r_{1}, \ldots, r_{k}]$ with $\deg(r_{i})=1$ for each $i$ and $\dim(\clM) =d$, then there exists ${\ds f_{\clM}(t) \in \ZZ[t, t^{-1}] }$ with  $f_{\clM}(1) \neq 0$ such that 
\[  \rmh_{\clM}(t) = \frac{f_{\clM}(t)}{(1-t)^{d}}. \] 
Let ${\ds f_{\clM}(t) = c_{l}t^{l} + c_{l+1}t^{l+1} + \cdots + c_{s}t^{s}}$.  Then we obtain the following. 
\[ \begin{array}{rcl}
{\ds \frac{f_{\clM}(t)}{(1-t)^{d}} } &=& {\ds \Big( c_{l}t^{l} + c_{l+1}t^{l+1} + \cdots + c_{s}t^{s}  \Big) \sum_{m=0}^{\infty} {{m+d-1}\choose{d-1}} t^{m}   }  \\ && \\
&=& {\ds \mbox{lower terms} + \sum_{n=s}^{\infty} \rmP_{\clM}(n) t^{n},}
\end{array} \]
where 
\[ \rmP_{\clM}(n) = c_{l}{{n-l+d-1}\choose{d-1}} + c_{l+1}{{n-l+d-2}\choose{d-1}} + \cdots + c_{s}{{n-s+d-1}\choose{d-1}}.\]
Then ${\ds \rmP_{\clM}(X) \in \QQ[X]}$ is a polynomial of degree $d-1$ and  ${\ds \rmH_{\clM}(n) = \rmP_{\clM}(n)}$ for all $n \geq s$. Since the polynomials ${\ds  {{X+i}\choose{i}}}$ with $i \in \NN$ form a $\QQ$--basis of $\QQ[X]$, we can rewrite
\[ \rmP_{\clM}(X)= \sum_{i=0}^{d-1} a_{i} {{X+i}\choose{i}} \;\; \mbox{with} \; a_{i} \in \QQ. \]
Moreover, by induction on $i$ and by using the fact that $\rmP_{\clM}(n) \in \ZZ$, we can prove that $a_{i} \in \ZZ$. This polynomial $\rmP_{\clM}(n)$ is called the {\em Hilbert polynomial} of $M$. By changing the index, we write
\[ \rmP_{\clM}(X) = \sum_{i=0}^{d-1} (-1)^{i} \rme_{i}(\clM) {{X+d-1-i}\choose{d-1-i}  }, \]
where $\rme_{i}(\clM)$'s are called the {\em Hilbert coefficients} of $\clM$. The {\em multiplicity} of $\clM$ is defined to be
\[ \rme(\clM) = \left\{  \begin{array}{ll}  \rme_{0}(\clM) \quad &\mbox{if} \;\; d >0, \vspace{0.1 in} \\ \l(\clM) &\mbox{if} \;\; d=0.   \end{array}  \right. \]

\medskip

Let $(R, \fkm)$ be a Noetherian local ring and $I$ an $\fkm$-primary ideal. Let $\G=\G(I)= \bigoplus_{n \geq 0} I^{n}/I^{n+1}$ be the associated graded ring of $I$. Then we can define the Hilbert function, Hilbert series, and Hilbert polynomial of $\G$ as above. That is,
\[ \rmH_{\G}(n) = \l(I^{n}/I^{n+1}), \quad \mbox{and} \quad \rmP_{\G}(X) = \sum_{i=0}^{d-1} (-1)^{i} \rme_{i}(\G) {{X+d-1-i}\choose{d-1-i}  }. \]
In some of the literature, these are known as the Hilbert function and the Hilbert polynomial of an ideal $I$. 
However, in recent publications, slightly revised formats for the Hilbert function and the Hilbert polynomial of an ideal $I$ are used more commonly. Thus, in this article, we will adopt these revised formats, which we now define. 

\begin{Definition}{\rm
Let $(R, \fkm)$ be a Noetherian local ring and $I$ an $\fkm$-primary ideal.  The {\em Hilbert function} of $I$ is the  function ${\ds \rmH_{I}: \NN_{0} \rar \NN_{0}}$ defined by  ${\ds \rmH_{I}(n) = \l(R/ I^{n+1})}$. The {\em Hilbert series} of $I$ is  the Laurent series $\rmh_{I}(t)=\sum_{n \geq 0} \rmH_{I}(n)t^n$.
}\end{Definition}

To see that  this function ${\ds \rmH_{I}}$ eventually agrees with a polynomial and  all but one of the coefficients are as before, we may assume that there exists an integer $u$ such that for all $n \geq u$ we have ${\ds \rmH_{\G}(n)= \l(I^{n}/I^{n+1}) = \rmP_{\G}(n)}$. 
Let ${\gla= \sum_{j=0}^{u-1} \l(I^{j}/I^{j+1})}$.  Then for all $n \geq u$, we obtain the following.

\[ \begin{array}{cl}
& {\ds  \l(R/I^{n+1})} \vspace{0.1 in} \\
=& {\ds  \gla + \sum_{j=u}^{n} \rmP_{\G}(j)   } \vspace{0.1 in} \\
=& {\ds \gla + \rme_{0}(\G) \sum_{j=u}^{n} {{j+d-1}\choose{d-1}} - \rme_{1}(\G) \sum_{j=u}^{n} {{j+d-2}\choose{d-2}} + \cdots + (-1)^{d-1} \rme_{d-1}(\G) (n-u+1)   }  \vspace{0.1 in} \\ 
= & {\ds  \rme_{0}(\G) \sum_{j=0}^{n} {{j+d-1}\choose{d-1}} - \rme_{1}(\G) \sum_{j=0}^{n} {{j+d-2}\choose{d-2}} + \cdots + (-1)^{d-1} \rme_{d-1}(\G)(n+1) + (-1)^{d} \rme_{d}   } \vspace{0.1 in} \\
=& {\ds \rme_{0}(\G){{n+d}\choose{d}} - \rme_{1}(\G){{n+d-1}\choose{d-1}} + \cdots + (-1)^{d-1} \rme_{d-1}(\G)(n+1) + (-1)^{d} \rme_{d}}.
\end{array}\]
Note that $(-1)^{d} \rme_{d}$ is the constant combining $\gla$ and all the constants resulting from changing the index of summation from $i=u$ to $i=0$.
 
\begin{Definition}{\rm 
Let ${\ds (R, \fkm)}$ be a Noetherian local ring of dimension $d>0$ and $I$ an $\fkm$--primary ideal. 
The {\em Hilbert polynomial} of $I$ is the unique polynomial ${\ds \rmP_{I}(X) \in \QQ[X]}$ of degree $d$ such that ${\ds \rmP_{I}(n) = \l(R/I^{n+1})}$ for all sufficiently large $n$. More specifically, 
\[ \rmP_{I}(X) = \sum_{i=0}^{d} (-1)^{i} \rme_{i}(I){{X+d-i}\choose{d-i}},\]
where $\rme_{i}(I)$'s are integers.  In this case, $\rme_{i}(I)$'s are called the {\em Hilbert coefficients} of $I$. In particular,  ${\ds \rme_{0}(I)}$ is called the {\em multiplicity} of $I$.
}\end{Definition}

Note that if $d \geq 1$, then $\rme_{i}(I)= \rme_{i}(\G(I))$ for all $i=0, \ldots d-1$. Though $\rmP_{I}(X)$ and $\rmP_{\G(I)}(X)$ have different binomial coefficients,  as far as the Hilbert coefficients are concerned, there is no difference except the existence of $\rme_{d}(I)$. 

\medskip

Now we present the notion of a reduction of an ideal, which was first introduced and exploited by Northcott and Rees \cite{NR54}. Let $J \subseteq I$ be ideals. Then $J$ is called a {\em reduction} of $I$ if there exists a nonnegative integer $r$ such that  $I^{r+1} = J I^{r}$. The smallest such integer is called the {\em reduction number} of $I$ relative to $J$ and is denoted by ${r_{J}(I)}$. The reduction number of $I$ is denoted by $r(I)$ and is defined as 
\[ r(I) = \min \{ r_{J}(I) \,\mid\, \mbox{$J$ is a minimal reduction of $I$} \}, \]
where minimal reductions are reductions which are minimal with respect to inclusion. Numerous questions revolve around the relationship between the Hilbert coefficients of $I$ and the reduction number of $I$. 

\medskip

Vasconcelos introduced Sally modules in \cite[Definition 2.1]{V94} to give a quick access to some results of Judith Sally connecting the depth properties of associated graded rings to extremal values of the multiplicities $\rme_{0}(I)$.  Vasconcelos' paper \cite{V94} had a profound impact on the field with numerous authors having cited the paper. Starting in the 1990's and continuing until today, authors cite the paper for its use of Sally modules, results relating Hilbert coefficients, particularly $\rme_{0}(I)$ and $\rme_{1}(I)$, and its examination of the interplay among ideal invariants including reduction numbers.

\begin{Definition}{\rm
Let $R$ be a Noetherian ring and let $I$ be an ideal in $R$. The {\em Rees algebra} $\clR(I)$ of $I$ is a subalgebra of the polynomial ring $R[t]$ defined by
\[ \clR(I) = R[It] = \bigoplus_{n \geq 0} I^{n} t^{n}. \] 
Let $J$ be a reduction of $I$. The {\em Sally module} $S_{J}(I)$ of $I$ with respect to $J$ is defined by the following exact sequence of finitely generated $\clR(J)$-modules.
\[ 0 \lar I \clR(J) \lar I \clR(I) \lar S_{J}(I) = \bigoplus_{n=1}^{\infty} I^{n+1} /IJ^{n} \lar 0.\]
}\end{Definition}

The following summarizes some key results regarding the dimension of the Sally module.

\begin{Theorem} Let $(R, \fkm)$ be a Noetherian local ring of dimension $d$ with infinite residue field. Let $I$ be an $\fkm$-primary ideal with a minimal reduction $J$. 
\begin{enumerate}[{\rm (1)}]
\item Suppose that $R$ is Cohen-Macaulay and ${\ds S_{J}(I) \neq 0}$. Then the dimension of $S_{J}(I)$ as an $\clR(J)$-module is $d$ \cite[Proposition 2.2]{V94}. 
\item Suppose that $R$ is Buchsbaum and that $I$ contains ${\ds (x_{1}, \ldots, x_{d-1}): \fkm }$, where ${\ds x_{1}, \ldots, x_{d-1}}$ are general elements of $J$. Then ${\ds S_{J}(I)}$ has dimension $d$ or $0$ \cite[Proposition 2.10]{C09}.
\item Suppose that $I=\fkm$. Then $S_{J}(\fkm)$ has dimension $d$ if and only if ${\ds \rme_{1}(\fkm) -\rme_{0}(\fkm) - \rme_{1}(J) +1 > 0}$ \cite[Theorem 2.1]{C09}.
\end{enumerate}
\end{Theorem}

\medskip

The notion of Hilbert coefficients of an ideal can be extended to those of a module.

\begin{Definition}\label{hcmod}{\rm
Let $(R, \fkm)$ be a Noetherian local ring of dimension $d >0$. Let $I$ be an $\fkm$-primary ideal and $M$ a finitely generated $R$-module of dimension $s \geq 1$. 
\begin{enumerate}[(1)]
\item The {\em Hilbert function of $M$ relative to $I$} is ${\ds \rmH_{I,M}(n) = \l(M/I^{n+1}M)}$.
\item The {\em Hilbert polynomial of $M$ relative to $I$} is 
\[ \rmP_{I, M}(n) = \sum_{i=0}^{s} (-1)^{i} \rme_{i}(I, M) {{n+s-i}\choose{s-i}}. \]
\item The integers $\rme_{i}(I, M)$ are called the {\em Hilbert coefficients} of $M$ relative to $I$. In particular, if $M=R$, then ${\ds \rme_{i}(I, R) = \rme_{i}(I)}$ for each $i$.
\item If $I=\fkm$, we write ${\ds \rme_{i}(\fkm, M) = \rme_{i}(M)}$ for each $i$. In particular,  ${\ds \rme_{0}(M)}$ is called the {\em multiplicity} of $M$.
\end{enumerate}
}\end{Definition}

The Rees algebra $\clR(I)$ and the associated graded ring $\G(I)$ of an ideal $I$ are often referred to as blowup algebras of $I$ and are based on the $I$-adic filtration ${\ds \{ I^{n} \mid n \in \NN_{0} \} }$. Now we give more general definitions of filtrations and their blow-up algebras.

\begin{Definition}\label{I-good}{\rm
Let $R$ be a Noetherian ring. A set of ideals 
\[ \clE=\{ I_{n} \mid n \in \NN_{0}, \; I_{0}=R, \; I_{1} \neq R, \; I_{n+1} \subseteq I_{n}, \; I_{n}I_{m} \subseteq I_{n+m} \} \]
is called a {\em filtration} in $R$. Let $I$ be an ideal of $R$. We say a filtration $\clE$ is an {\em $I$-good filtration} if ${\ds I I_{n} \subseteq I_{n+1}}$ for all $n$ and there exists $k$ such that for all $n \geq k$, ${\ds II_{n} = I_{n+1}}$. 
}\end{Definition}

Given a filtration ${\ds \clE=\{ I_{n} \}_{n \in  \NN_{0}} }$, we can define the Rees algebra ${\ds \clR( \clE)}$ and the  associated graded ring ${\ds \G(\clE) }$ as the following.
\[ \clR( \clE) = \bigoplus_{n=0}^{\infty} I_{n} t^{n}, \hspace{0.3 in} \G(\clE) = \bigoplus_{n=0}^{\infty} I_{n}/I_{n+1}. \] In particular, $\clR(\clE)$  is a finitely generated $\clR(I)$-module if and only if  $\clE$ is a good $I$-filtration.  
If $(R, \fkm)$ is a Noetherian local ring, $I$ is an $\fkm$-primary ideal, and $\clE$ is a good $I$-filtration, then we can define the Hilbert function ${\ds \rmH_{\clE}(n)}$, Hilbert series ${\ds \rmh_{\clE}(t)}$, Hilbert polynomial ${\ds \rmP_{\clE}(X)}$, and Hilbert coefficients ${\ds \rme_{i}(\clE)}$ of the filtration $\clE$ by using the definitions of ${\ds \rmH_{\clM}(n), \rmh_{\clM}(t), \rmP_{\clM}(X), \rme_{i}(\clM) }$.  

\medskip

A particular $I$-good filtration involving integral closures has been of great interest. The {\em integral closure} $\cl{I}$ of $I$ is the set of all elements $y$ that are integral over $I$, that is, satisfy a polynomial equation of the form
\[ y^{m} + a_{1} y^{m-1} + \cdots+ a_{j} y^{m-j} + \cdots + a_{m} =0, \]
where ${\ds a_{j} \in I^{j}}$.  

\begin{Definition}\label{normalfilt}{\rm 
Let $R$ be a Noetherian ring and $I$ an ideal in $R$.
\begin{enumerate}[(1)]
\item The {\em normal filtration} of $I$ is ${\ds \clN = \{ \cl{I^{n}} \mid  n \in  \NN_{0}, \; \cl{I^{0}}=R \} }$.
\item The Rees algebra of the normal filtration is called the {\em normal Rees algebra} of $I$ and is denoted by ${\ds \cl{\clR}(I)}$. That is, \[  \cl{\clR}(I) =\bigoplus_{n=0}^{\infty} \cl{I^{n}}t^{n}.   \]
In particular,  ${\ds \cl{\clR}(I)}$ is the integral closure of $\clR(I)$ in the polynomial ring $R[t]$.
\item The associated graded ring of the normal filtration is called the {\em normal associated graded ring} of $I$ and is denoted by ${\ds \cl{\G}(I)}$. That is, \[  \cl{\G}(I) =\bigoplus_{n=0}^{\infty} \cl{I^{n}}/\cl{I^{n+1}}.   \]
\end{enumerate}
}\end{Definition}

Recall that a Noetherian local ring $(R, \fkm)$ is said to be {\em analytically unramified} if its $\fkm$-adic completion $\widehat{R}$ is reduced. If $R$ is analytically unramified, then the normal filtration $\clN$ of $I$ is an $I$-good filtration.

\begin{Definition}\label{normalHilb}{\rm 
Let $(R, \fkm)$ be an analytically unramified Noetherian local ring. Let $I$ be  an $\fkm$-primary  ideal in $R$.
The {\em normal Hilbert function} of $I$ is  the function ${\ds \cl{\rmH}_{I}: \NN_{0} \rar \NN_{0}}$ defined by  ${\ds \cl{\rmH_{I}}(n) = \l(R/ \cl{I^{n+1}})}$. The {\em normal Hilbert polynomial} of $I$ is the unique polynomial ${\ds \cl{\rmP}_{I}(X) \in \QQ[X]}$ of degree $d$ such that ${\ds \cl{\rmP}_{I}(n) = \l(R/\cl{I^{n+1}})}$ for all sufficiently large $n$. More specifically, 
\[ \cl{\rmP}_{I}(X) = \sum_{i=0}^{d} (-1)^{i} \cl{\rme}_{i}(I){{X+d-i}\choose{d-i}},\]
where $\cl{\rme}_{i}(I)$'s are integers.  In this case, $\cl{\rme}_{i}(I)$'s are called the {\em normal Hilbert coefficients} of $I$. 
}\end{Definition}

\medskip

In the subsequent sections, when we state the theorems, we will use the notation given in this section. Though the presentations may seem different from the original statements in the referenced papers, we believe that using consistent notation throughout this article may give the readers a clearer understanding.

\bigskip

\section{Early Results: Bounds Involving Hilbert Coefficients}\label{early}

In this section, we focus on several of Vasconcelos' most influential early works regarding Hilbert coefficients. One of his most influential and heavily cited papers was \cite{V94}. In this work, Vasconcelos extended techniques that had been developed for ideals generated by $d$-sequences (see (\ref{d-seq}) for the definition of $d$-sequences) to the arithmetical study of more general algebras through the use of a minimal reduction of an ideal that had well-behaved Koszul homology. He also studied analytic spread, Cohen-Macaulayness of the Rees algebra $\clR(I)$, the number of generators of certain prime ideals, and the behavior of the Hilbert polynomial of $I$ when $I$ is primary.  A key tool used in this analysis was the Sally module. In particular, the following result  shows that if $(R, \fkm)$ is a  Cohen-Macaulay local ring, then the Hilbert function of an $\fkm$-primary ideal can be written in terms of the multiplicity of the ideal and the length of the corresponding component of the Sally module.

\begin{Theorem}\label{V94-3-1}{\rm \cite[Proposition 3.1]{V94}}
Let $(R, \fkm)$ be a Cohen-Macaulay local ring of dimension $d$, with infinite residue field, and let $I$ be an $\fkm$-primary ideal. If $J$ is a minimal reduction of $I$, then
\[ \rmH_{I}(n) = \rme_{0}(I) {{n+d}\choose{d}} + \left(\l(R/I) - \rme_{0}(I)\right){{n+d-1}\choose{d-1}} - \l\left(I^{n+1}/IJ^{n}\right). \] 
\end{Theorem}

\begin{proof} Since ${\ds J^{n+1} \subset IJ^{n} \subset I^{n+1}}$, we have 
\[ \l( R/I^{n+1}) = \l(R/J^{n+1}) - \l(IJ^{n}/J^{n+1}) - \l(I^{n+1}/IJ^{n}). \]
Since $J$ is a reduction of $I$, we have ${\ds \rme_{0}(I)=\rme_{0}(J)}$. Since $J$ is a minimal reduction of an $\fkm$-primary ideal in a Cohen-Macaulay ring, $J$ is generated by a regular sequence of length $d$. Then ${\ds \rme_{0}(J) = \l(R/J)}$ and 
\[ \l( R/J^{n+1}) = \rme_{0}(I) {{n+d}\choose{d}},\]
and
\[ \l( IJ^{n}/J^{n+1}) = \l( I/J)  {{n+d-1}\choose{d-1}} = ( \rme_{0}(I) - \l(R/I) ) {{n+d-1}\choose{d-1}} . \]
Then the assertion follows.
\end{proof}

This result has several interesting implications regarding the Hilbert function of the Sally module. The relation above can be viewed in reverse, giving information on the Hilbert polynomial of the Sally module in terms of the Hilbert polynomial of the ideal. Let $R, I$, and $J$ be as in Theorem~\ref{V94-3-1}. Then the 
Hilbert polynomial of ${\ds S_{J}(I)}$ is 
\[ \begin{array}{rcl}
{\ds \rmP_{S_{J}(I)}(n) } &=& {\ds \rme_0(I) {{n+d}\choose{d}} + \left(\l(R/I) - \rme_0(I)\right){{n+d-1}\choose{d-1}} - \rmP_{I}(n) } \vspace{0.12 in} \\
&=& {\ds  \rme_{0}(I) {{n+d}\choose{d}} + \left(\l(R/I) - \rme_{0}(I)\right){{n+d-1}\choose{d-1}} } \vspace{0.08 in} \\
&& {\ds - \rme_{0}(I) {{n+d}\choose{d}} + \rme_{1}(I) {{n+d-1} \choose{d-1}} - \rme_{2}(I) {{n+d-2}\choose{d-1}} + \cdots - (-1)^{d} \rme_{d}(I)} \vspace{0.12 in} \\
&=& {\ds \left( \l(R/I) - \rme_{0}(I) + \rme_{1}(I) \right){{n+d-1}\choose{d-1}} - \rme_{2}(I) {{n+d-2}\choose{d-1}} + \cdots - (-1)^{d} \rme_{d}(I)}
\end{array}  \]

In particular, the two Hilbert polynomials ${\ds \rmP_{I}(n)}$ and ${\ds \rmP_{S_{J}(I)}(n)}$  have similar growth, with a degree shift of one, and beyond the initial terms, the coefficients of the Hilbert polynomials will match up to a degree shift.  These results are summarized in the corollary below.

\begin{Corollary}{\rm \cite[Corollaries 3.2, 3.3, 3.4]{V94}}\label{Sally bounds}
Let $(R, \fkm)$ be a Cohen-Macaulay local ring of dimension $d>0$, $I$ an $\fkm$-primary ideal, and $J$ a minimal reduction of $I$.  Let $S_{n}$ denote the $n$-th component of $S_{J}(I)$. 
\begin{enumerate}[{\rm (1)}]
\item If  $S_{J}(I) \neq 0$, then the Hilbert function $\l(S_n)$ has the growth of a polynomial of degree $d-1$.
\item If $S_{J}(I) \neq 0$ and $s_0, s_1, \ldots, s_{d-1}$ are the coefficients of the Hilbert polynomial of $S_{J}(I)$, then 
\[ \rme_{1}(I) = \rme_{0}(I) - \l(R/I) + s_0, \quad \mbox{and} \quad  \rme_{i+1} = s_{i} \;\; \mbox{for} \;\; i \geq 1. \]
\item ${\ds \rme_{0}(I) - \rme_{1}(I) \leq \l(R/I)}$.
\item If ${\ds \rme_{0}(I) - \rme_{1}(I) = \l(R/I)}$, then ${\ds S_{J}(I)=0}$.
\end{enumerate}
\end{Corollary}

The relationships in Corollary~\ref{Sally bounds} are part of Vasconcelos' broader study of Hilbert polynomials. The difference ${\ds \rme_{0}(I) - \rme_{1}(I)}$ in particular appears in multiple works. The result below is a start in this direction, giving a consequence of the inequality in Corollary~\ref{Sally bounds}-(3)  being one off from an equality.

\begin{Proposition}{\rm \cite[Proposition 3.5]{V94}}
Let $(R, \fkm)$ be a Cohen-Macaulay local ring with infinite residue field $k$. Let $I$ be an $\fkm$-primary ideal, and $J$ a minimal reduction of $I$. If \[ \rme_{0}(I)- \rme_{1}(I) = \l(R/I) -1,\]  then $S_{J}(I)$ is isomorphic to an ideal of $k[T_1, \ldots, T_d]$.
\end{Proposition}

As a consequence of this result, Vasconcelos recovered a result of Sally showing that the difference ${\ds \rme_0(I) - \rme_{1}(I)}$ is closely related to the reduction number of $I$.

\begin{Theorem}{\rm \cite[Theorem 3.6]{V94}}
Let $(R, \fkm)$ be a Cohen-Macaulay local ring with infinite residue field and let $I$ be an $\fkm$-primary ideal. If
${\ds \rme_{0}(I)- \rme_{1}(I) = \l(R/I) -1}$ and  ${\ds \rme_{2}(I) \neq 0}$, then $I$ has reduction number $2$.
\end{Theorem}

The preceding results focus on the coefficients of the Hilbert polynomial of an $\fkm$-primary ideal and the related coefficients of the Hilbert polynomial of the Sally module. The components of the Sally module are quotients focused on a (minimal) reduction. Another module that encodes properties of an ideal through quotients is the associated graded ring
$\G=\G(I)= \bigoplus_{n \geq 0} I^{n}/I^{n+1}$.  
As discussed in Section~\ref{background}, $\rme_{0}(I)= \rme_{0}(\G)$ and for $i < d$, ${\ds \rme_{i}(I)=\rme_{i}(\G)}$. This yields another tool for studying multiplicities of ideals through Hilbert coefficients. In \cite{RVV01}, Rossi, Valla, and Vasconcelos used these tools to study which ideals have maximal Hilbert coefficients and the interplay between $\rme_{0}(I)$ and $\rme_{1}(I)$.  

\begin{Theorem}\label{RVV01-2-2}{\rm \cite[Theorem 2.2, Remark 2.3]{RVV01}}
Let $(R, \fkm)$ be a Noetherian local ring of dimension $d \geq 1$ and let $I$ be an $\fkm$-primary ideal. 
Let ${\ds \rmh_{\G}(t)}$ denote the Hilbert series of the associated graded ring $\G=\G(I)$ of $I$. 
If $J$ is an ideal generated by a system of parameters in $I$, then 
\[ \rmh_{\G}(t) \leq \frac{\l(R/I) + \l(I/J) t}{(1-t)^d}.\]
If the equality holds, then $\G(I)$ is Cohen-Macaulay, $R$ is Cohen-Macaulay, and $\rme_{0}(I)=\l(R/J)$.  
\end{Theorem}

\medskip

A direct relationship between $\rme_{0}(I)$ and $\rme_{1}(I)$ can be found in a $1$-dimensional Cohen-Macaulay ring.

\begin{Theorem}{\rm \cite[Proposition 3.1]{RVV01}}\label{one-dim}
Let $(R, \fkm)$ be a $1$-dimensional Cohen-Macaulay local ring and let $I$ be an $\fkm$-primary ideal. If ${\ds \rme_{0}(I) \neq \rme_{0}(\fkm)} $, then 
\[ \rme_{1}(I) \leq {{\rme_{0}(I) -2}\choose{2}}.\]
\end{Theorem}

In dimension one, using superficial elements allows one to bound the degree of the numerator of the Hilbert series in terms of $\rme_{0}(I)$ and $\rme_{1}(I)$ (see \cite[Proposition 3.2]{RVV01}). Generalizing to higher dimensional rings yields a more general, albeit slightly weaker, bound. 

\medskip

Recall that if $(R, \fkm)$ is Cohen-Macaulay and the Sally module $S_{J}(I)$ is not trivial, then the multiplicity of $S_{J}(I)$ is ${\ds  \rme_{1}(I) - \rme_{0}(I) + \l(R/I)}$ (Corollary~\ref{Sally bounds}). 

\begin{Theorem}{\rm \cite[Theorem 3.3]{RVV01}}\label{Sally-series}
Let $(R, \fkm)$ be a Cohen-Macaulay ring of dimension $d \geq 1$ and let $I$ be an $\fkm$-primary ideal. 
Let ${\ds \rmh_{I}(t)}$ denote the Hilbert series of $I$. Let ${\ds s= \rme_{1}(I) - \rme_{0}(I) + \l(R/I)}$. Then
\[ \rmh_{I}(t) \leq \frac{ \l(R/I) + \big( \rme_{0}(I) - \l(R/I) -1   \big)t + t^{s+1}}{(1-t)^{d+1}}. \]
\end{Theorem}

It is worth noting that \cite[Theorem 3.3]{RVV01} also shows that, if the Hilbert series ${\ds \rmh_{I}(t)}$ of $I$  has precisely the maximal form, then the associated grade $\G(I)$ has depth at least $d-1$.  That is,  $\G(I)$ is either Cohen-Macaulay or almost Cohen-Macaulay.

\bigskip

\section{Multiplicities}\label{mult}

A large part of Vasconcelos' study of Hilbert functions focused on what he referred to as degree functions, particularly multiplicities, which will be the focus of this section. In studying multiplicities and Hilbert coefficients, Vasconcelos used the interplay of multiple graded algebras, including the Rees algebra, the associated ring, the Sally module, the conormal module, the fiber cone, and the extended Rees algebra, to obtain information about the reduction number of an ideal and to generalize interesting results beyond the Cohen-Macaulay case. 

\medskip

Let $(R, \fkm)$ be a Noetherian local ring of dimension $d >0$ and let $M$ be a finitely generated $R$-module of dimension $s \geq 1$. The definition of the multiplicity $\rme_{0}(M)$ is given in Definition~\ref{hcmod}.  In many Vasconcelos' papers, the multiplicity of $M$ is also denoted by $\deg(M)$. If $\dim(M)=d$, then the multiplicity can also be written as
\[ \rme_{0}(M) = \rme_{0}(\fkm, M) =  \lim_{n \rar \infty} \frac{d!}{n^d}\l(M/\fkm^{n+1} M). \]
If $I$ is an $\fkm$-primary ideal, then
\[ \rme_{0}(I) = \rme_{0}(I, R) = \lim_{n \rar \infty} \frac{d!}{n^d}\l(R/I^{n+1}). \]

\medskip

When working with ideals of positive dimension, one technique is to use a filtration where ideally the factors are sufficiently well-behaved, for instance Cohen-Macaulay. This will lead to additional factorial terms but allows a study of multiplicities and how they relate to other ideal invariants. 

\medskip

By a theorem of Lech \cite{L60}, if $I$ is $\fkm$-primary, then
\[ \rme_0(I) \leq d! \, \rme_0(R) \,  \l(R/\cl{I}), \]
where $\cl{I}$ is the integral closure of $I$. Vasconcelos extended this to the more general case through the use of a reduction of $I$. Recall that an ideal $I$ is said to be {\em equimultiple} if its analytic spread (See Definition~\ref{spfiber})  equals the height of $I$. 
 
\begin{Theorem}{\rm \cite[Proposition 2.3]{V03-1}}
Let $(R, \fkm)$ be a Cohen-Macaulay local ring with infinite residue field and let $I$ be an equimultiple ideal of height $g$. If $J$ is a minimal reduction of $I$, then
\[ \rme_0(R/J) \leq g! \, \rme_0 (R/\cl{I}) \, \rme_0(R).\] 
\end{Theorem}

One application of studying multiplicities is that one then has tools that can be used to study other invariants. For instance, the multiplicity can be used to provide a bound on the number of generators of $I$.

\begin{Theorem}{\rm \cite[Proposition 2.1]{V03-2}}
Let $(R, \fkm)$ be a Cohen-Macaulay local ring and let $I$ be a Cohen-Macaulay ideal of height $g$. Then
\[ \nu(I) \leq \rme_0(R) + (g-1) \rme_0(R/I). \]
\end{Theorem}

When $I$ is not necessarily Cohen-Macaulay, then using any of the generalized degree functions, a ``correction term" can be used to generalize this bound (see \cite{DGV}). Additional bounds on the number of generators of $I$ in terms of the multiplicity of the ring include those in the following theorem.

\begin{Theorem}{\rm (\cite[Theorem 2.2]{V03-2}, \cite{S76})}
Let $(R, \fkm)$ be a Cohen-Macaulay local ring of dimension $d>0.$
\begin{enumerate}[{\rm (1)}]
\item If $I$ is an $\fkm$-primary ideal of nilpotency index $t$, then
\[ \nu(I) \leq t^{d-1} \rme_{0}(R) +d-1. \]
\item If $I$ is a Cohen-Macaulay ideal of height $g>0$, then
\[ \nu(I) \leq \rme_{0}(R/I)^{g-1}\rme_{0}(R) +g-1.\]
\end{enumerate}
\end{Theorem}

When some control over $I$ is known, these results can be improved. For example, if the reduction number of $I$ is small, or $R/I$ is Cohen-Macaulay, then one has the bounds below.

\begin{Theorem}{\rm \cite[Theorem 2.4, Propositions 2.5 and 3.1]{V03-2}}
Let $(R, \fkm)$ be a Cohen-Macaulay local ring of dimension $d$. 
\begin{enumerate}[{\rm (1)}]
\item If $I$ is a Cohen-Macaulay ideal of height $2$ and the type of $R/I$ is $r$, then
\[ \nu(I) \leq (r+1) \rme_{0}(R). \]
\item If $R$ has type $r$ and $I$ is an equimultiple Cohen-Macaulay ideal of height $g \geq 1$, then
\[ \nu(I) \leq  (r+1) \rme_{0}(R/I) +g-1. \]
\item If $I$ is an $\fkm$-primary ideal, then
\[ \nu(I) \leq \rme_{0}(I) +d-1. \]
\item If $I$ is an $\fkm$-primary ideal and $I \subset \fkm^{2}$, then
\[ \nu(I) \leq \rme_{0}(I) \leq d! \, \l(R/\cl{I}) \, \rme_{0}(R).\]
\end{enumerate}
\end{Theorem}

\medskip

Another approach  is to focus on minimal reductions of $I$. Starting first with a minimal reduction of an equimultiple Cohen-Macaulay ideal, the following theorem provides bounds on the number of generators of $I$ in terms of a minimal reduction and the radical of $I$.

\begin{Theorem}{\rm \cite[Proposition 4.1]{V03-2}}
Let $(R, \fkm)$ be a Cohen-Macaulay local ring and  $I$ an equimultiple Cohen-Macaulay ideal of height at least one. Let $J$ be a  minimal reduction of $I$. 
\begin{enumerate}[{\rm (1)}]
\item If $\sqrt{I}$ is a Cohen-Macaulay ideal, then
\[ \nu(I) \leq \rme_{0}(R/J) + (g-1) \rme_{0}(R/\sqrt{I}).\]
\item If in addition ${\ds I \subseteq \left(\sqrt{I}\right)^2}$, then
\[ \nu(I) \leq \rme_{0}(R/J) - \rme_{0}(R/\sqrt{I}).\]
\end{enumerate}
\end{Theorem}

When one assumes more control over the minimal reduction, the results can be fine-tuned. Suppose $J_{0} \subset I$ is generated by a regular sequence of $g-1$ elements and is part of a minimal reduction of $I$.  Then using that $I/J_{0}$ is maximal Cohen-Macaulay of rank 1 over $R/J_{0}$, which is also Cohen-Macaulay, by \cite{V03-1},  one has that
\[ \nu(I) \leq \nu(I/J_{0}) + \nu(J_{0}) \leq \rme_{0}(R/J_{0}) + g -1.\]
As an application of this technique, consider the case where $I$ is equimultiple and $J=(J_{0},x)$ is a minimal reduction of $I$. If $R$ is Gorenstein, then $\rme_{0}(R/J_{0}) \leq \rme_{0}(R/J)$ and by \cite{V03-1}, we have
\[ \nu(I) \leq \rme_{0}(R/J) + g -1.\]
Continuing in this setting using the standard short exact sequence (and noting that since $R$ is Gorenstein and $(J:I)/J$ is the canonical module of $R/I$) yields
\[ \rme_{0}(R/J) = \rme_{0}(R/I) + \rme_{0}(R/(J:I)). \]
Extending these two methods to powers of ideals results in an additional need for binomial coefficients \cite{V03-1}.
\[ \nu(I^n) \leq \rme_{0}(R/J_{0}^n) + \nu(J_{0}^n) \leq \rme_{0}(R/J_{0}) {{n+g-1}\choose{g-1}}+{{n+g-2}\choose{g-2}}.\]
These bounds lead to related results involving additional invariants of ideals. Recall that $r(I)$ denotes the reduction number of an ideal. 

\begin{Theorem}{\rm \cite[Theorem 3.1]{V03-1}}
Let $(R, \fkm)$ be a Cohen-Macaulay local ring with infinite  residue field and let $I$ be an ideal of height $g > 0$. If $I$ is normally Cohen-Macaulay, then
\[ r(I) \leq g \cdot g! \, \rme_{0}(R/\cl{I}) \, \rme_{0}(R) -2g+1.\]
\end{Theorem}

Under somewhat stronger assumptions, the bound on the reduction number can be simplified.

\begin{Theorem}{\rm \cite[Theorem 4.10]{V03-1}}
Let $(R, \fkm)$ be a Gorenstein local ring of dimension $d$ and let $I$ be a perfect Gorenstein ideal of dimension $2$ with $R/I$ normal. Suppose that the residue field $R/\fkm$ has characteristic zero. Then
\[ r(I) \leq (d-1)\, \rme_{0}(\G(I)) -4d +5.\]
\end{Theorem}

\bigskip

\section{Normalization of Ideals and Rings}\label{normal}

Vasconcelos used a variety of filtrations when studying Hilbert polynomials. This section focuses on his use of good filtrations, such as normal filtrations. 
 An ideal $I$ in a Noetherian ring $R$ is said to be {\em normal} if  ${\ds \cl{I^n} = I^n}$ for all $n \geq 1$. Equivalently, $I$ is normal if and only if ${\ds \cl{\clR}(I) = \clR(I)}$ (See Definition~\ref{normalfilt}).  Using $\cl{\clR}(I)$ as the {\em normalization} of $I$ leads to interesting results regarding the Hilbert coefficients. 

\medskip

Let $(R, \fkm)$ be an analytically unramified Noetherian local ring and $I$ an $\fkm$-primary ideal. 
In \cite{PUV05}, Polini, Ulrich, and Vasconcelos used a sequence of modules connecting $\clR(I)$ and $\cl{\clR}(I)$ to deduce information about Hilbert coefficients.  In particular, they focused on bounding $\cl{\rme}_1(I)$. One connection of this work to that of the prior section is through the minimal reductions of $I$. Recall that $J \subseteq I$ is a reduction of $I$ if and only if 
$J$ and $I$ have the same integral closure.  The {\em Brian\c{c}on-Skoda}  number $b(I)$ of $I$ is the smallest integer $b$ such that ${\ds \cl{I^{n+b}} \subseteq J^{n}}$ for every $n$ and every reduction $J$ of $I$.

\medskip

One goal in the study of the normalizations is to find effective bounds on $\cl{\rme}_1(I)$, which necessarily yield bounds on $\rme_{1}(I)$ as well. A key technique is to pass information through a chain of modules as mentioned in the preceding paragraph. Since $\rme_{1}(\tratto)$ is non-decreasing on such a chain, one can use $\rme_{1}(\tratto)$ to obtain numerical criteria for normality. An example of the applications of this technique is contained in the following theorem and its corollaries.

\begin{Theorem}\label{closure chain}{\rm \cite[Theorem 2.2]{PUV05}}
Let $(R, \fkm)$ be an analytically unramified local Cohen-Macaulay ring of positive dimension with infinite residue field and let $I$ be an $\fkm$-primary ideal. Let $A$ and $B$ be distinct  graded $R$-subalgebras of $R[t]$ with 
\[ \clR(I) \subseteq A \subsetneq B \subseteq \cl{\clR}(I).\] 
Suppose that  $A$ satisfies the condition $(\rmS_{2})$ of Serre.  Then
\[ 0 \leq \rme_{1}(I) \leq \rme_{1}(A) < \rme_{1}(B) \leq \cl{\rme}_1(I) \leq b(I)\rme_{0}(I). \]
\end{Theorem}

As an immediate consequence of Theorem~\ref{closure chain}, for such rings, $\cl{\rme}_1(I)$ is the maximal length of a chain of modules satisfying $(\rmS_{2})$ between $\clR(I) $ and $\cl{\clR}(I)$. Also, it follows that  $I$ is normal if and only if $\clR(I)$ satisfies $(\rmS_{2})$ and $\cl{\rme}_{1}(I)= \rme_{1}(I)$  (See \cite[Corollary 2.4, Corollary 2.5]{PUV05}).
Using the Brian\c{c}on-Skoda number produces an improved bound on $\cl{\rme}_{1}(I)$ in terms of $\rme_{0}(I)$.

\begin{Theorem}{\rm \cite[Theorem 3.2, Corollary 3.4]{PUV05}}
Let $(R, \fkm)$ be a  reduced local  analytically unramified Cohen-Macaulay ring of dimension $d>0$ with infinite residue field. Let $I$ be an $\fkm$-primary ideal. Then
\[ \cl{\rme}_{1}(I) \leq b(I) \min \left\{ \frac{t}{t+1} \rme_{0}(I),\;  \rme_{0}(I) - \l(R/\cl{I}) \right\}.\]
If in addition $R$ is regular, then
\[ \rme_{1}(I) \leq \cl{\rme}_{1}(I) \leq (d-1) \min \left\{ \frac{ \rme_{0}(I)}{2}, \;  \rme_{0}(I) - \l (R/\cl{I})\right\}.\]
\end{Theorem}

The preceding theorem naturally gives rise to a need to bound $b(I)$. Using the multiplicy of a module defined via a regular element of the Jacobian ideal $\Jac_{K}(R)$, such a bound can be found. 

\begin{Proposition}{\rm \cite[Proposition 3.7]{PUV05}}
Let $k$ be a perfect field, let $(R, \fkm)$ be a reduced local Cohen-Macaulay $k$-algebra essentially of finite type of dimension $d > 0$, and let $\delta \in \Jac_{k}(R)$ be a non-zerodivisor. Then for any $\fkm$-primary ideal $I$,
\[ b(I) \leq d-1 + \rme_{0}(I+\delta R/\delta R).\]
\end{Proposition}

When additional information is known about the depth of the various graded algebras associated to an ideal, in particular when ${\ds \cl{\clR}(I) }$ or $\G(I)$ are Cohen-Macaulay or almost Cohen-Macaulay, then the relationships between the invariants are conveniently intertwined with the multiplicity. 

\medskip

Let $(R, \fkm)$ be an analytically unramified Noetherian local ring and $I$ an ideal of $R$. The {\em normalization index} of $I$ is the smallest nonnegative integer $\sigma=\sigma(I)$ such that ${\ds \cl{I^{n+1}} = I \, \cl{I^{n}} }$ for all $n \geq \sigma$. The {\em generation index} of $I$ is the smallest nonnegative integer $\tau=\tau(I)$ such that ${\ds \cl{\clR}(I) = R[ \cl{I}t, \ldots, \cl{I^{\tau}} t^{\tau} ]}$ \cite[Definition 2.1]{PUVV19}.  If $I$ is an $\fkm$-primary ideal, $\sigma(I)$ is bounded above by by a polynomial of degree $d+1$ in $\tau(I)$ whose leading coefficient is a multiple of $\rme_{0}(I)$ (See \cite[Proposition 2.3]{PUVV19}.). When $R$ is a polynomial ring over a field of characteristic $0$, the bound becomes particularly nice. 

\begin{Theorem}{\rm \cite[Theorem 2.4]{PUVV19}}
Let $R=k[x_1, \ldots, x_d]$ be a polynomial ring over a field of characteristic zero and let $I$ be a homogeneous ideal that is $(x_1, \ldots, x_d)$-primary. Then
\[ \sigma(I) \leq ( \rme_{0}(I) -1){{\tau(I)+1}\choose{2}} - (\rme_{0}(I) -2){{\lfloor \frac{\tau(I)}{2}\rfloor +1}\choose{2}}.\]
\end{Theorem}

Combining the use of reductions and filtrations to gain information about invariants of ideals naturally leads to the Sally module, particularly when the invariants in question are related to Hilbert functions. Let $\clE=\{ I_{n} \}$ be an $I$-good filtration in $R$ (See Definition~\ref{I-good}). Let $J$ be a reduction of $I$. The {\em Sally module of the filtration $\clE$ relative to $J$} is defined as 
\[ S_{J}(\clE) = \bigoplus_{n \geq 1} I_{n+1}/I_{1} J^{n}.  \]
Suppose that $(R, \fkm)$ is a Cohen-Macaulay ring of dimension $d$ and $I$ is an $\fkm$-primary ideal. Then the Hilbert series of $\clE$ and $S=S_{J}(\clE)$ are related as shown in \cite{PUVV19}, i.e., 
\[ \rmh_{\clE}(t) =  {\frac{\l(R/I) + \l(I/J)\cdot t}{(1-t)^{d+1}}} - \rmh_{S}(t). \]
It follows that $\rme_{i+1}(\clE) = \rme_{i}(S)$ for $i \geq 1$, which leads to a relation among the Hilbert coefficients of the filtration.

\begin{Theorem}{\rm (\cite[Corollary 3.5]{PUVV19}, \cite[Corollary 2]{M89})}
Let $(R, \fkm)$ be  a Cohen-Macaulay local ring of dimension $d>0$, and $I$ an $\fkm$-primary ideal. 
Let $\clE=\{I_{n}\}$ be a $I$-good filtration.
Let $J$ be a reduction of $I$ and ${\ds S=S_{J}(\clE)}$ the Sally module of $\clE$ relative to $J$. 
Suppose that the associated graded ring ${\ds \G(\clE)}$  is Cohen-Macaulay and that the Hilbert series of $S$ can be written as ${\ds \rmh_{S}(t) = g(t)/(1-t)^{d} }$, where $g(t)$ is a polynomial of degree at most $4$. Then 
\[ \rme_{2}(\clE) \geq \rme_{3}(\clE) \geq \rme_{4}(\clE) \geq \rme_{5}(\clE). \]
\end{Theorem}

The Hilbert coefficients of a filtration can also be used to bound the number of generators of the Rees algebra of the filtration.  We assume $J$ is a fixed minimal reduction associated to an ideal $I$ in what follows. Typical applications of the following theorem are those to the $I$-adic filtration $\{I^n\}$ and the normal filtration $\{\cl{I^n}\}$.

\begin{Theorem}{\rm \cite[Theorem 3.7]{PUVV19}}\label{filtration}
Let $(R, \fkm)$ be a Cohen-Macaulay local ring of dimension $d \geq 1$. Let ${\ds \clE=\{I_{n} \}}$ be filtration in $R$ with $I_{1}=I$.  Let ${\ds A= \clR(J)}$ and ${\ds B= \clR(\clE)}$. 
\begin{enumerate}[{\rm (1)}]
\item If ${\ds  \depth( \G(\clE))  \geq d-1}$, then  ${\ds \nu_{A}( B/A) \leq \rme_{1}(\clE)}$. 

\item If ${\ds  \depth (\G(\clE))  = d}$, then ${\ds \nu_{A}(B) \leq \rme_{0}(\clE)}$. 

\item If  $B$ is Cohen-Macaulay, then  ${\ds \nu_{A}( B/A) \leq   \nu(I/J) + \max\{0, d-2\}\l(I_2/JI)}$.   
\end{enumerate}
\end{Theorem}

The connection between multiplicities and integral closure is more general than the theorems above, which focused on the Rees algebra. Indeed, multiplicity can be used to measure the complexity of the integral closure in a more general setting. 

\medskip

Let $k$ be a field and let $A$ be a finitely generated $k$-algebra. Suppose that $A$ is reduced and equidimensional of dimension $d$ and let $T=k[x_1, \ldots , x_d]$ be a Noether normalization of $A$. The rank of $A$ over $T$ is the dimension of the $K$-vector space $A \otimes_{T} K$ where $K$ is the quotient field of $T$. The {\em multiplicity} $\deg(A)$ of $A$ is defined as the least rank of $A$ as $T$-module for all possible Noether normalizations. If $k$ is infinite and $A$ is  a standard graded $k$-algebra, then the rank of $A$ is independent of $T$ and $\deg(A)=\rme_{0}(A)$.  This viewpoint on multiplicity provides a way to use multiplicity to bound invariants of ring extensions, which can then be applied in the case of integral closures.

\medskip

Let $A \subset B$  be affine $k$-algebras with the following presentations.
\[ A= k[x_{1}, \ldots, x_{n}]/I  \hookrightarrow B=k[y_{1}, \ldots, y_{m}]/J.\]
The least $n$, among all such presentations of $A$, is called the {\em embedding dimension} of $A$ and is denoted by ${\ds \edim(A)}$. The {\em embedding degree} of $B$ is the smallest value of ${\ds \max\{ \deg(y_{i}) \}}$ achieved among all presentations and is denoted by $\edeg(B)$ (See \cite{UV04}). Notice that if $A$ is a reduced ring which is a finitely generated $k$-algebra, then the integral closure $\cl{A}$ is also an affine $k$-algebra.
The numbers  $\edim(\cl{A})$ and $\edeg(\cl{A})$ are important measures for the complexity of computing the integral closure $\cl{A}$.

\medskip

\begin{Theorem}{\rm \cite[Theorem 2.1]{UV04}}
Let $k$ be a field.
Let $A$ be a reduced and equidimensional finitely generated $k$-algebra of dimension $d$.  Let $A \subset B$ be a finite, birational ring extension and assume that $B$ is Cohen-Macaulay.
\begin{enumerate}[{\rm (1)}]
\item ${\ds \nu_A(B) \leq \deg(A)}$
\item ${\ds \edim(B) \leq \deg(A) + d - 1}$
\item If $k$ is perfect, $A$ is a standard graded $k$-algebra, and $A \subset B$ is an extension of graded rings, then  ${\ds \edeg(B) \leq  \max \{1, \rme_{0}(A) -2 \}}$.
\end{enumerate}
\end{Theorem}

\begin{Corollary}
Let $A$ be a reduced and equidimensional finitely generated $k$-algebra. Suppose that the integral closure $\cl{A}$ is Cohen-Macaulay. Then 
\[ \edim(\cl{A}) \leq \deg(A) + d-1. \]
If $k$ is perfect, $A$ is a standard graded $k$-algebra, and  $\rme_{0}(A) \geq 3$, then ${\ds \edeg(\cl{A}) \leq \rme_{0}(A) -2 }$.
\end{Corollary}

For birational extensions, there are additional relationships between multiplicities when the base field has characteristic zero. The bounds are improved when additional information is known about the multiplicity of the base ring or the depth of the extension.  See \cite[Lemma 3.1, Theorem 3.2, Corollary 3.4]{UV04} for more general and detailed statements. Here we include the results regarding only the integral closure $\cl{A}$. 

\begin{Theorem}{\rm \cite[Theorem 3.2, Corollary 3.4]{UV04}}\label{embedding}
Let $k$ be a field of characteristic $0$. Let $A$ be a reduced and equidimensional standard graded $k$-algebra of dimension $d$. 
\begin{enumerate}[{\rm (1)}]
\item If ${\ds \depth_{A}( \cl{A}) \geq d-1}$, then ${\ds \nu_{A}(\cl{A}) \leq (\rme_{0}(A) -1)^2 +1}$
\item If ${\ds \depth_{A}( \cl{A}) \geq d-1}$, then ${\ds \edim( \cl{A} ) \leq (\rme_{0}(A) -1)^2 + d +1}$
\item If $d=3$ and ${\ds \rme_{0}(A) \geq 3}$, then ${\ds \edim(\cl{A}) \leq (\rme_{0}(A) -1)^2+2}$. 
\end{enumerate}
\end{Theorem}

If $A$ is not a standard graded ring, then the above results can be extended in some situations. For instance, let $A=k[x_1, \ldots,x_{d+1}]/(f)$ where $f$ is a (not necessarily homogeneous) square-free polynomial. Then bounds on the multiplicity of a birational extension locally at primes $\fkp$ of the polynomial ring that are analogous in form, but somewhat weaker, can be found (See \cite[Lemma 4.2]{UV04}). 

\bigskip

\section{Vasconcelos Rings and Modules}\label{Vrings}

In this section, we present one of Vasconcelos' most widely studied conjectures and related results.
Let $(R, \fkm)$ be a Noetherian local ring of dimension $d \geq 1$. Let $I$ be an $\fkm$-primary ideal and $M$ a finitely generated $R$-module of dimension $s \geq 1$. Recall from Definition~\ref{hcmod} that the Hilbert polynomial of $M$ relative to $I$ is 
\[ \rmP_{I, M}(n) = \sum_{i=0}^{s} (-1)^{i} \rme_{i}(I, M) {{n+s-i}\choose{s-i}}. \]
The coefficient $\rme_{1}(I, M)$ is called the {\em Chern coefficient} of $I$ relative to $M$. 

\medskip

A parameter ideal for $M$ is an ideal ${\ds Q=(x_{1}, \ldots, x_{s} ) \subseteq \fkm}$ in $R$ such that ${\ds \l(M/QM) < \infty}$. It is known that a parameter ideal $Q$ for $M$ satisfies $\rme_{1}(Q, M) \leq 0$ (\cite[Corollary 2.4]{Chern3}, \cite[Corollary 2.3]{Chern5}). 
If $M$ is Cohen-Macaulay, then $\rme_{1}(Q, M)=0$ for every parameter ideal $Q$ for $M$. This leads to the question of when the converse holds true so that  the negativity of $\rme_{1}(Q, M)$ can be used as an expression of the lack of Cohen-Macaulayness of $M$. Vasconcelos conjectured the following at the conference in Yokohama in March 2008.

\begin{Conjecture}\label{Ch1-C3-1}{\rm {\bf (Vasconcelos' Vanishing Conjecture \cite{Chern1})}  Let $(R, \fkm)$ be a Noetherian local ring of positive dimension. Suppose that $R$ is unmixed. Then ${\ds \rme_{1}(Q) =0}$ for a parameter ideal $Q$ if and only if $R$ is Cohen-Macaulay.   
}\end{Conjecture}

In \cite{Chern1} and \cite{Chern2}, Vasconcelos and several authors found various classes of rings for which the conjecture was proven to be true. The following summarizes those results. 

\begin{Theorem}
Let $(R, \fkm)$ be a Noetherian local ring of dimension at least $2$. Suppose that $R$ satisfies one of the following conditions{\rm:}
\begin{enumerate}[{\rm (i)}]
\item  There is an embedding $0 \rar R \rar E \rar C \rar 0$, where $E$ is a finitely generated maximal Cohen-Macaulay $R$-module \cite[Theorem 3.1]{Chern1}{\rm;}
\item  The ring $R$ is an integral domain essentially of finite type over a field \cite[Theorem 3.2]{Chern1}{\rm;}
\item  The ring $R$ is an integral domain and a homomorphic image of a Cohen-Macaulay local ring \cite[Theorem 3.3]{Chern2}{\rm;}
\item  The ring $R$ is unmixed, equidimensional, and a homomorphic image of a Gorenstein local ring \cite[Corollary 3.7]{Chern2}{\rm;}
\item  The ring $R$ is a universally catenary integral domain containing a field \cite[Theorem 4.4]{Chern2}.
\end{enumerate}
Then $\rme_{1}(Q) =0$ for any parameter ideal $Q$ if and only if $R$ is Cohen-Macaulay.
\end{Theorem}

Ghezzi, Goto, Hong, Ozeki, Phuong, and Vasconcelos settled the Conjecture~\ref{Ch1-C3-1} affirmatively and then extended the result to more general case of modules.

\begin{Theorem}\label{Ch3-2-1}{\rm \cite[Theorem 2.1]{Chern3}}
Let $(R, \fkm)$ be a Noetherian local ring of positive dimension. Let $Q$ be a parameter ideal. Then the following are equivalent\,{\rm:}
\begin{enumerate}[{\rm (i)}]
\item $R$ is Cohen-Macaulay\,{\rm ;}
\item $R$ is unmixed and $\rme_{1}(Q) =0$\,{\rm ;}
\item $R$ is unmixed and $\rme_{1}(Q) \geq 0$.
\end{enumerate}
\end{Theorem}

\begin{Theorem}\label{Ch5-3-1}{\rm \cite[Theorem 3.1]{Chern5}}
Let $(R, \fkm)$ be a Noetherian local ring and $M$ a finitely generated $R$-module of dimension at least $2$. Let $Q$ be a parameter ideal for $M$. Suppose that $M$ is unmixed. Then $M$ is a Cohen-Macaulay $R$-module if and only if ${\ds \rme_{1}(Q, M) =0}$.
\end{Theorem}

\begin{proof} It is enough to prove that if $M$ is not Cohen-Macaulay, then ${\ds \rme_{1}(Q, M) < 0}$. We may assume that $R$ is complete with an infinite residue field and $\dim(R) = \dim(M) =d$. Let $S$ be a Gorenstein local ring of dimension $d$ with a surjective homomorphism $\varphi: S \rar R$. There exists a parameter ideal $\fkq$ of $S$ such that ${\ds \fkq R = Q}$ and ${\ds \rme_{1}(Q, M) = \rme_{1}(\fkq, M) }$ \cite[Lemma 3.1]{Chern2}. Since $M$ is unmixed, by \cite[Proposition 2.5]{Chern5}, there exists an exact sequence ${\ds 0 \rar M \rar S^{n} \rar C \rar 0}$ of $S$-modules.  Let $y$ be a superficial element for $\fkq$ with respect to $M$ such that $y$ is part of a minimal generating set of $\fkq$. We may assume that $y$ is a regular element on $M$. Then we obtain the following exact sequence of $S$-modules. 
\[ 0 \lar \Tor_{1}^{S} (S/(y), C) \lar M/yM \stackrel{\zeta}{\lar}  S^{n}/yS^{n} \lar C/yC \lar 0. \]
Let ${\ds T=  \Tor_{1}^{S} (S/(y), C)}$, ${\ds \cl{M}=M/yM}$, and ${\ds N= \mbox{Im}(\zeta)}$ and consider the short exact sequence:
\begin{equation}\label{ses1}\tag{$\mho$}
 0 \lar T \lar \cl{M} \lar N \lar 0.
 \end{equation}
We use induction on $d$ to prove that if $M$ is not Cohen-Macaulay, then $\rme_{1}(\fkq, M) <0$. Let $d=2$ and  $\fkq=(y, z)$. By applying the Snake Lemma to the following
\[\begin{CD}
0 @>>> T \cap z^{n}  \cl{M} @>>> z^{n}\cl{M} @>>> z^{n} N @>>> 0 \\ 
&& @VVV  @VVV  @VVV   \\
0 @>>> T  @>>> \cl{M} @>>> N @>>> 0,
\end{CD} \] we obtain that, for sufficiently large $n$, 
\[ \l( \cl{M}/z^{n} \cl{M} ) = \l (T)  + \l (N/ z^{n} N). \]
By computing the Hilbert polynomials, we have
\[ \rme_{1}(\fkq, M) = \rme_{1}( \fkq/(y), \cl{M} ) = - \l (T) < 0. \]
Suppose that $d \geq 3$. The exact sequence $(\mho)$ induces 
\[ \rme_{1}(\fkq, M) = \rme_{1}( \fkq/(y), \cl{M} ) = \rme_{1}( \fkq/(y), N). \]
Note that $N$ is an unmixed $S/(y)$-module and ${\ds \dim(S/(y)) = d-1}$. By an induction argument, it is enough to show that $N$ is not Cohen-Macaulay. 

\medskip

\noindent Suppose that $N$ is Cohen-Macaulay. Let $\fkn$ denote the maximal ideal of $S/(y)S$.  From  the exact sequence $(\mho)$, we obtain the following long exact sequence:
\[ 0 \rar  H^{0}_{\fkn}(T) \rar H^{0}_{\fkn}(\cl{M}) \rar H^{0}_{\fkn}(N) \rar H^{1}_{\fkn}(T) \rar H^{1}_{\fkn}(\cl{M}) \rar H^{1}_{\fkn}(N). \]
Since $N$ is Cohen-Macaulay and $\dim(N)= d-1 \geq 2$, we have ${\ds H^{0}_{\fkn}(N) = 0 = H^{1}_{\fkn}(N)}$. Also, since $T$ is a torsion module, we have ${\ds T=  H^{0}_{\fkn}(T)}$ and ${\ds  H^{1}_{\fkn}(T) =0 }$. Therefore,
\[ T \simeq H^{0}_{\fkn}(\cl{M}), \quad \mbox{and} \quad H^{1}_{\fkn}(\cl{M}) =0. \]
From the exact sequence ${\ds 0 \rar M \stackrel{\cdot y}{\lar} M \rar \cl{M} \rar 0}$, we obtain 
\[ 0 \lar T \simeq H^{0}_{\fkn}(\cl{M}) \lar H^{1}_{\fkn}(M) \stackrel{\cdot y}{\lar}  H^{1}_{\fkn}(M)  \lar H^{1}_{\fkn}(\cl{M}) =0. \]
Since ${\ds H^{1}_{\fkn}(M)}$ is finitely generated \cite[Corollary 2.6]{Chern5} and ${\ds H^{1}_{\fkn}(M) = y H^{1}_{\fkn}(M)}$, we have ${\ds H^{1}_{\fkn}(M) =0}$. Then $T=0$. Thus, the exact sequence ($\mho$) induces $\cl{M} \simeq N$. Therefore, ${\ds \cl{M}=M/yM}$ is Cohen-Macaulay. Since $y$ is regular on $M$, $M$ is Cohen-Macaulay. This is a contradiction. 
\end{proof}

In honor of Vasconcelos' Vanishing Conjecture, which paves a new way to understand a ring through Hilbert coefficients, the first five authors of \cite{Chern3} introduced a new terminology called Vasconcelos rings and modules.

\begin{Definition}\label{Vasconcelos}{\rm 
Let $(R, \fkm)$ be a Noetherian local ring. A finitely generated $R$-module $M$ is called a {\em Vasconcelos module} if either $\dim_{R}(M)=0$ or ${\ds \rme_{1}(Q, M) =0}$ for some parameter ideal $Q$ for $M$. If $R$ itself is a Vasconcelos module, then it is called a Vasconcelos ring.
}\end{Definition}

In particular, a Cohen-Macaulay module is a Vasconcelos module. Moreover, Theorem~\ref{Ch5-3-1} proves that an unmixed Vasconcelos module is a Cohen-Macaulay module. There exists a non-Cohen-Macaulay Vasconcelos ring given in  \cite[Example 3.6]{Chern3}. Though the definition of a Vasconcelos module requires only a vanishing of the Chern coefficient of {\em some} parameter ideal, it turns out a Vasconcelos module guarantees a vanishing of the Chern coefficient of {\em every} parameter ideal.

\begin{Proposition}\label{Ch5-3-7}{\rm (\cite[Theorem 3.4]{Chern3}, \cite[Theorem 3.7]{Chern5}) }
Let $(R, \fkm)$ be a Noetherian local ring and $M$ a finitely generated $R$-module of dimension at least $2$. Then $M$ is a Vasconcelos module if and only if $\rme_{1}(Q, M)=0$ for every parameter ideal $Q$ for $M$.
\end{Proposition}

Many fundamental properties of a Vasconcelos ring regarding flat base changes, Rees algebras, associated graded rings, and sequentially Cohen-Macaulayness  can be found in \cite[Section 3]{Chern3} and can be extended to a Vasconcelos module. The following is a list of a few notable properties of a Vasconcelos ring. 

\begin{Proposition} Let $(R, \fkm)$ be a Vasconcelos ring of dimension $d \geq 1$.
\begin{enumerate}[{\rm (1)}]
\item If $x$ is a regular element in $R$, then $R/xR$ is a Vasconcelos ring \cite[Corollary 3.5]{Chern3}.
\item Let ${\ds S=R[X_{1}, \ldots, X_{n}]}$ be a polynomial ring. Then $S_{\fkP}$ is a Vasconcelos ring for every prime ideal $\fkP$ containing $\fkm S$ \cite[Corollary 3.9]{Chern3}. 
\item Suppose that $R$ is a homomorphic image of a Cohen-Macaulay ring and that ${\ds \dim(R/\fkq) =d}$ for every minimal prime ideal $\fkq$ of $R$. Then ${\ds R_{\fkp}}$ is a Vasconcelos ring for every prime ideal $\fkp$ of $R$ \cite[Proposition 3.11]{Chern3}. 
\end{enumerate}
\end{Proposition}

\bigskip

\section{The Chern Coefficients}\label{Chern}

Proposition~\ref{Ch5-3-7} leads to a question about when the Chern coefficient $\rme_{1}(Q, M)$ is independent of the choice of parameter ideals $Q$ for $M$. In this section, we will discuss  Vasconcelos' conjectures and results of Vasconcelos and his coauthors relating to situations where either $e_1(Q,M)$ is independent or can take on only finitely many values. As a first example of such modules, consider Buchsbaum modules. A finitely generated $R$-module $M$ of dimension $s$ is said to be {\em Buchsbaum} if 
\[ \l(M/QM) - \rme_{0}(Q, M) = \sum_{i=0}^{s-1} {{s-1}\choose{i}} \l(H^{i}_{\fkm}(M)) \]
for every parameter ideal $Q$ for $M$. By \cite[Lemma 2.4, Theorem 5.4]{GN03} and \cite[Korollar 3.2]{S79}, if $M$ is a Buchsbaum module, then  $\rme_{1}(Q, M)$ is constant and independent of the choice of parameter ideals $Q$ for $M$. The following theorem proved that the converse is true for a unmixed module.

\begin{Theorem}\label{Ch5-5-4}{\rm \cite[Theorem 5.4]{Chern5}}
Let $(R, \fkm)$ be a Noetherian local ring. Let $M$ be a finitely generated unmixed $R$-module of dimension $s \geq 2$. Then $M$ is a Buchsbaum $R$-module if and only if $\rme_{1}(Q, M)$ is constant and independent of the choice of parameter ideals $Q$ for $M$. In this case, we have
\[ \rme_{1}(Q,M) = - \sum_{i=1}^{s-1} {{s-2}\choose{i-1}} \l(H^{i}_{\fkm}(M)) \]
for every parameter ideal $Q$ for $M$.
\end{Theorem}

Let ${\ds \Lambda(M) = \{ \rme_{1}(Q, M) \mid Q \; \mbox{is a parameter ideal for} \; M \} }$.  Theorem~\ref{Ch5-5-4} shows that a unmixed module $M$ is Buchsbaum if and only if $|\Lambda(M)|=1$.  A finitely generated $R$-module $M$ is said to be {\em generalized Cohen-Macaulay} if $M_{\fkp}$ is Cohen-Macaulay for every non-maximal prime ideal $\fkp$. By \cite[Lemma 2.4]{GN03} and  \cite[Corollary 2.3]{Chern5}, if $M$ is generalized Cohen-Macaulay, then $\Lambda(M)$ is a finite set. The following theorem proved that the converse is true for a unmixed module.

\begin{Theorem}\label{Ch5-4-2}{\rm \cite[Proposition 4.2 and Theorem 4.5]{Chern5}}
Let $(R, \fkm)$ be a Noetherian local ring. Let $M$ be a finitely generated unmixed $R$-module of dimension $s \geq 2$. Then $M$ is a generalized Cohen-Macaulay $R$-module if and only if $\Lambda(M)$ is a finite set.  In this case, we have
\[  - \sum_{i=1}^{s-1} {{s-2}\choose{i-1}} \l(H^{i}_{\fkm}(M)) \leq  \rme_{1}(Q,M) \leq 0 \]
for every parameter ideal $Q$ for $M$.
\end{Theorem}

\medskip

Goto and Ozeki \cite{GO11} characterized the generalized Cohen-Macaulayness of a (not necessarily unmixed) ring in terms of the finiteness of sets of Hilbert coefficients.

\begin{Theorem}\label{GO11-1-1}{\rm \cite[Theorem 1.1]{GO11}}
Let $(R, \fkm)$ be a Noetherian local ring of dimension $d \geq 2$. Then  $R$ is a generalized Cohen-Macaulay ring if and only if each set \[ \Lambda_{i}(R) = \{\rme_{i}(Q) \mid Q \; \mbox{is a parameter ideal for} \; R   \} \] is finite for all $1 \leq i \leq d$.  
\end{Theorem}

\medskip

A ring $R$ is called {\em sequentially Cohen-Macaulay} if there exists a filtration of ideals ${\ds \left\{ I_{t} \mid t=0, \ldots, n \right\}}$ such that $I_{0}=(0)$, ${\ds I_{n}=R}$, ${\ds I_{t} \subsetneq I_{t+1}}$,  ${\ds \h(I_{t}) < \h(I_{t+1}) }$ and $I_{t+1}/I_{t}$ are Cohen-Macaulay for all $t=0, 1, \ldots, n-1$. It is known that $R$ is a Cohen-Macaulay ring if and only if $R$ is an unmixed sequentially Cohen-Macaulay ring. Ozeki, Truong, and Yen \cite{OTY22} characterized sequentially Cohen-Macaulay rings (which are not unmixed) in terms of the Hilbert coefficients. 

\begin{Theorem}\label{OTY22-4-1}{\rm \cite[Theorem 4.1]{OTY22} }
Let $R$ be a homomorphic image of a Cohen-Macaulay ring with ${\ds \dim(R) =d \geq 1}$. Then $R$ is sequentially Cohen-Macaulay if and only if there exists a distinguished parameter ideal ${\ds Q \subseteq \fkm^{g(R)}}$ such that 
\[ (-1)^{d-j} \Big( \rme_{d-j+1}(Q:\fkm) - \rme_{d-j+1}(Q) \Big) \leq r_{j}(R) \quad \mbox{\rm for all $2 \leq j \in \{ \dim(R/\fkp) \mid \fkp \in \Ass(R) \} $ }, \]
where $g(R)$ denotes the $g$-variant of $R$ {\rm (\cite[Definition 3.2]{OTY22})} and ${\ds r_{j}(R) = \l( 0 :_{H^{j}_{\fkm}(R)} : \fkm) }$. 
\end{Theorem}

\medskip

Theorem~\ref{Ch5-4-2} shows that uniform bounds for $\rme_{1}(Q, M)$ exist for special classes of modules.
Vasconcelos posed a conjecture regarding such uniform bounds in more general cases. The definition of $I$-good filtration is given in Definition~\ref{I-good}.

\begin{Conjecture}\label{Ch1-Conj3}{\rm {\bf (Vasconcelos' Uniformity Conjecture \cite{Chern1}) } Let $(R, \fkm)$ be a Noetherian local ring and $I$ an $\fkm$-primary ideal of $R$.  Let $\clA$ be a $I$-good filtration. Then there exist functions $f_{l}(I)$ and $f_{u}(I)$ defined with extended degree over $R$ such that ${\ds f_{l}(I) \leq \rme_{1}(\clA) \leq f_{u}(I)}$.
}\end{Conjecture}

Vasconcelos settled the conjecture for the existence of a lower bound using the homological degree \cite[Definition 2.8]{V98-1}. Let $(R, \fkm)$ be a Noetherian local ring of dimension $d>0$ that is a homomorphic image of a Gorenstein local ring $S$. Let $I$ be an $\fkm$-primary ideal of $R$. Then the {\em homological degree} of $R$ with respect to $I$ is defined as 
\[ \hdeg_{I}(R) = \rme_{0}(I) + \sum_{i=1}^{d} {{d-1}\choose{i-1}} \hdeg_{I} \left( \Ext^{i}_{S}(R, S) \right). \]
He also established the existence of an upper bound for special cases by using the  Brian\c{c}on-Skoda theorem on perfect fields. 

\begin{Proposition}\label{Ch1-C7-7}{\rm \cite[Corollary 7.7, Theorem 6.1]{Chern1}}
Let $(R, \fkm)$ be a Noetherian local ring of dimension $d \geq 1$ and $I$ an $\fkm$-primary ideal.  Let $\clA$ be an $I$-good filtration.
\begin{enumerate}[{\rm (1)}]
\item Suppose that $R$ is a homomorphic image of a Gorenstein local ring. Then 
\[ \rme_{1}(I) \geq - \hdeg_{I}(R) + \rme_{0}(I). \]
\item Suppose that $R$ is a generalized Cohen-Macaulay integral domain that is essentially of finite type over a perfect field.
Then \[ \rme_{1}(\clA) < (d-1) \rme_{0}(I) + \rme_{0}( (I, \delta)/(\delta) ) - \l(T), \]
where $\delta$ is a nonzero element of the Jacobian ideal of $R$.
\end{enumerate}
\end{Proposition}

Let $I$ be an $\fkm$-primary ideal and let $Q$ be a parameter ideal such that ${\ds \cl{Q} = \cl{I}}$. The lower bound for $\rme_{1}(I)$ given in  Proposition~\ref{Ch1-C7-7} shows that we have 
\[ 0 \geq  \rme_{1}(Q) \geq  - \hdeg_{Q}(R) + \rme_{0}(Q)  = - \hdeg_{I}(R) + \rme_{0}(I)   \]
Therefore the set ${\ds \Omega(I) = \{ \rme_{1}(Q) \mid  \mbox{$Q$ is a parameter ideal and ${\ds \cl{Q}=\cl{I}}$ } \} }$ is a finite set. This leads to the question of when $\Omega(I)$ contains only one element. This is not always true even when $I=\fkm$.(See \cite[Example 6.3]{Chern6}.) However in some special cases, there is an affirmative answer to the question.

\begin{Theorem}{\rm \cite[Theorem 4.4]{Chern6} }
Let $(R, \fkm)$ be a Noetherian local ring of positive dimension and $I$ an $\fkm$-primary ideal. Suppose that, for every minimal reduction $Q$ of $I$, the Rees algebra ${\ds \clR(Q) }$ is a graded generalized Cohen-Macaulay ring. Then ${\ds \rme_{1}(Q)}$ is independent of the choice of minimal reduction $Q$ of $I$.
\end{Theorem}

As Vasconcelos' Conjecture~\ref{Ch1-Conj3} was settled in several cases,  Goto and Ozeki (\cite{GO11}) explored the uniform bounds for $\rme_{2}(Q)$, where $Q$ is a parameter ideal.

\begin{Theorem}\label{GO11-3-6}{\rm \cite[Theorems 3.2 and 3.6]{GO11}}
Let $(R, \fkm)$ be a Noetherian local ring of dimension $d>0$. Let ${\ds Q=(a_{1}, \ldots, a_{d})}$ be a parameter ideal for $R$. 
\begin{enumerate}[{\rm (1)}]
\item Suppose that $d=2$, ${\ds \depth(R) >0}$, and $a_{1}$ is superficial with respect to $Q$. Then
\[ - \l \left( H^{1}_{\fkm}(R) \right) \leq \rme_{2}(Q) \leq 0. \]
Moreover, ${\ds \rme_{2}(Q) =0}$ if and only if ${\ds a_{1}, a_{2}}$ forms a $d$-sequence in $R$.
\item Suppose that $d \geq 3$ and that $R$ is a generalized Cohen-Macaulay ring. Then 
\[ - \sum_{j=2}^{d-1} {{d-3}\choose{j-2}}\l \left(H^{j}_{\fkm}(R) \right)  \leq \rme_{2}(Q) \leq \sum_{j=1}^{d-2} {{d-3}\choose{j-1}}\l \left(H^{j}_{\fkm}(R) \right). \]
\end{enumerate}
\end{Theorem}

Recall that the normal Hilbert coefficients are denoted by ${\ds \cl{\rme}_{i}(I) }$ (See Definition~\ref{normalHilb}). 

\begin{Conjecture}\label{Conj2}{\rm {\bf (Vasconcelos' Positivity Conjecture \cite{Chern1})} Let $(R, \fkm)$ be  
an analytically unramified Noetherian local ring of positive dimension. Let $I$ be an $\fkm$-primary ideal in $R$. Then we have ${\ds \cl{\rme}_{1}(I) \geq 0}$. 
}\end{Conjecture}

Goto, Hong, and Mandal settled the Positivity Conjecture affirmatively.

\begin{Theorem}\label{e1bar-1-1}{\rm \cite[Theorem 1.1]{e1bar} }
Let $(R, \fkm)$ be an analytically unramified unmixed local ring of positive dimension. Then ${\ds \cl{\rme}_{1}(I) \geq 0 }$ for every $\fkm$-primary ideal $I$.
\end{Theorem}

\begin{proof} We outline the proof. We may assume that $R$ is complete. If $d=1$, then ${\ds \cl{\rme}_{1}(I) = \l(\cl{R}/R) \geq 0}$. Suppose that $d \geq 2$. Then 
\[ \cl{\rme}_{1}(I) \geq \sum_{\fkp \in {\tiny \Ass(R)}} \l_{R} \left( S(\fkp) / \fkm_{S(\fkp)} \right) \cdot \cl{\rme}_{1} (I S(\fkp) ), \]
where $S(\fkp)$ is the integral closure ${\ds \cl{R/\fkp} }$. Therefore, it is enough to prove that ${\ds \cl{\rme}_{1} (I S(\fkp) ) \geq 0}$ for each ${\ds \fkp \in \Ass(R)}$. We prove this by induction on $d$.  Let $d=2$. Then, since $S(\fkp)$ is Cohen-Macaulay, we get
\[ \cl{\rme}_{1} (I S(\fkp) )  \geq  \rme_{1} (I S(\fkp) )  \geq 0. \]
Suppose $d \geq 3$. By passing to $S(\fkp)$, we may assume that $R$ is a normal complete local ring.  Let ${\ds I=(a_{1}, \ldots, a_{n})}$, where ${\ds n = \nu(I)}$. Let 
\[ T=R[z_{1}, \ldots, z_{n}], \quad \fkq=\fkm T, \quad x = \sum_{i=1}^{n} a_{i}z_{i}, \quad D=T/x T, \quad I'=I T_{\fkq}, \quad  D'=D_{\fkq}\]
where $z_{1}, \ldots, z_{n}$ are indeterminates  over $R$.
Using the technique given in the proof of \cite[Theorem 2.1]{HU14}, it can be shown that 
\[  \cl{\rme}_{1}(I) = \cl{\rme}_{1}(I') = \cl{\rme}_{1}(ID').\]
Since ${\ds \dim(D') = d-1}$, the assertion follows by the induction hypothesis. 
\end{proof}

\cite[Example 2.1]{e1bar} shows that it is necessary to assume that $R$ is unmixed in Theorem~\ref{e1bar-1-1}.
As a consequence of Theorem~\ref{e1bar-1-1}, we obtain the following.

\begin{Corollary}\label{e1bar-2-2}{\rm \cite[Corollary 2.2]{e1bar} }
Let $(R, \fkm)$ be an analytically unramified unmixed local ring of positive dimension. Let $Q$ be a parameter ideal.  Suppose that ${\ds \cl{\rme}_{1}(Q) = \rme_{1}(Q) }$. Then $R$ is a regular local ring and $Q$ is normal.
\end{Corollary}

By the nonpositivity of $\rme_{0}(Q)$ and Theorem~\ref{e1bar-1-1}, the assumptions in Corollary~\ref{e1bar-2-2} imply that ${\ds \cl{\rme}_{1}(Q) =0 }$. Thus, it is reasonable to expect that the vanishing of \cite{e1bar} may play a role in the normality of the ideal $Q$. 

\begin{Question}\label{e1barzero}{\rm
Let $(R, \fkm)$ be an analytically unramified unmixed local ring of positive dimension. Let $I$ be an $\fkm$-primary ideal. Is it true that, if ${\ds \cl{\rme}_{1}(I) =0 }$, then $\cl{R}$ is a regular ring and $I \cl{R}$ is normal?
}\end{Question}

Goto, Hong, and Mandal gave a positive answer to Question~\ref{e1barzero}  in some special cases.

\begin{Theorem}{\rm \cite[Theorem 3.1 and Corollary 3.2]{e1bar}}
Let $(R, \fkm)$ be an analytically unramified local ring of positive dimension $d$. Let $I$ be an $\fkm$-primary ideal.
Suppose $R$ satisfies one of the following conditions.
\begin{enumerate}[{\rm (i)}]
\item There exists an overring $S$ of $R$ such that $S$ is a finitely generated $R$-module with ${\ds \dim_{R}(S/R) <d }$ and ${\ds \depth_{R}(S) =d}$.
\item The ring $R$ is unmixed and $d=2$. 
\end{enumerate}
If ${\ds \cl{\rme}_{1}(I) =0 }$, then $\cl{R}$ is a regular ring and $I \cl{R}$ is normal.
\end{Theorem}

\bigskip

\section{Euler Characteristics and $j$-Transforms}\label{Euler}

In this section, we discuss how Hilbert functions, multiplicities, and Chern coefficients can be generalized to invariants that provide comparable information when the module $M$ is not unmixed or the ideal $I$ is not $\fkm$-primary. 

\medskip

Let $(R, \fkm)$ be a Noetherian local ring and $M$ a finitely generated $R$-module of dimension $s >0$. 
Let ${\ds \bfx=x_{1}, \ldots, x_{s}}$ be a system of parameters for $M$. We denote the $j$th Koszul homology of ${\ds \bfx}$ with coefficients in $M$ by ${\ds H_{j}(\bfx, M)= H_{j}(x_{1}, \ldots, x_{s}, M)}$.  The {\em first Euler characteristic} of $M$ relative to $\bfx$ is
\[  \chi_{1}(\bfx, M) = \sum_{j \geq 1} (-1)^{j-1} \l \left( H_{j} (\bfx, M)  \right).  \]
When $M=R$, we write ${\ds \chi_{1}(\bfx)= \chi_{1}(\bfx, R)}$ and ${\ds H_{j}(\bfx)= H_{j}(\bfx, M)}$.
Let $Q=(\bfx)$ be an ideal in $R$. By a classical result of Serre \cite{Serrebook}, we have 
\[  \chi_{1}(\bfx, M) = \l(M/QM) - \rme_{0}(Q, M). \]  In analogy to 
\[ \Lambda(M) = \{ \rme_{1}(Q, M) \mid Q \; \mbox{is a parameter ideal for} \; M \},\] we define
\[\Xi(M) = \{\chi_1(\bfx, M) \mid  \bfx   \; \mbox{is a system of parameters for} \; M \}.\]
Then the properties of the set $\Xi(M) $ mirror the properties of the set $\Lambda(M)$ listed in previous section without assuming $M$ is unmixed.

\begin{Theorem}\label{Euler1}{\rm  (\cite[Appendix II]{Serrebook}, \cite{SV86}, \cite{CST78})}
Let $(R, \fkm)$ be a Noetherian local ring  and $M$ a finitely generated $R$-module of dimension $s >0$. Let ${\ds \bfx =x_{1}, \ldots, x_{s}}$ be a system of parameters for $M$. 
\begin{enumerate}[{\rm (1)}]
\item $\chi_{1}(\bfx, M) \geq 0$.
\item $M$ is Cohen-Macaulay if and only if $\chi_{1}(\bfx, M) = 0$.
\item $M$ is Buchsbaum if and only if $\Xi(M)$ contains only one element.
\item $M$ is generalized Cohen-Macaulay if and only if $\Xi(M)$ is a finite set.
\end{enumerate}
\end{Theorem}

We say ${\ds \bfx=x_{1}, \ldots, x_{m}}$ is a partial system of parameters  of $M$ if  ${\ds \dim(M)= m + \dim(M/(\bfx)M)}$. A partial system $\bfx$ of parameters is said to be {\em amenable} to $M$ if the Koszul homology modules $H_{i}(\bfx, M)$ have finite length for all $i >0$.  When $\bfx=x_{1}, \ldots, x_{m}$ is amenable to $M$, we have
\[ \chi_{1}(x_{1}, \ldots, x_{m}, M) = \sum_{j=1}^{n} (-1)^{j-1} \l(H_{j}(x_{1}, \ldots, x_{m}, M)). \]

\begin{Definition}\label{d-seq}{\rm 
Let $R$ be a Noetherian ring and $M$ a finitely generated $R$-module. Let ${\ds x_{1}, \ldots, x_{m}}$ be  elements in $R$.
\begin{enumerate}[(1)]
\item  A sequence ${\ds \bfx= x_{1}, \ldots, x_{m}}$  is called a {\em $d$-sequence} relative to $M$ if 
\[ (x_{1}, x_{2}, \ldots, x_{i})M :_{M} x_{i+1}x_{k} = (x_{1}, x_{2}, \ldots, x_{i})M :_{M} x_{k} \]
for $i=0, \ldots, m-1$ and $k \geq i+1$.  

\item A sequence ${\ds \bfx= x_{1}, \ldots, x_{m}}$ is called a {\em strong $d$-sequence} relative to $M$ if the sequence ${\ds x_{1}^{r_{1}}, \ldots, x_{m}^{r_{m}} }$ is a $d$-sequence relative to $M$ for every sequence $r_{1}, \ldots, r_{m}$ of positive integers.

\item A sequence ${\ds \bfx= x_{1}, \ldots, x_{m}}$ is called a {\em  $d^{+}$-sequence} relative to $M$ if it is a strong $d$-sequence relative to $M$ in any order.
\end{enumerate}
}\end{Definition}

If ${\ds \bfx=x_{1}, \ldots, x_{s}}$ is a system of parameters of $M$ that is also a $d$-sequence relative to $M$, then ${\ds x_{1}, \ldots, x_{i}}$ is amenable for every $i=1, \ldots, s$. Given the similar roles of $\rme_{1}(Q, M)$ and $\chi_{1}(\bfx, M)$ as predictors of the Cohen-Macaulay property, it is natural to consider a direct comparison between them.  

\begin{Theorem}\label{Ch7-3-7}{\rm \cite[Theorem 3.7 and Corollary 3.8]{Chern7}}
Let $(R, \fkm)$ be a Noetherian local ring of dimension $d \geq 2$. Let ${\ds \bfx=x_{1}, \ldots, x_{d}}$ be a system of parameters that is a $d$-sequence in $R$. Let $Q=(\bfx)$. Then, for all $1 \leq i \leq d$, we have
\[ (-1)^{i} \rme_{i}(Q) = \chi_{1} \left(x_{1}, \ldots, x_{d-i}, x_{d-i+1} \right) - \chi_{1} \left(x_{1}, \ldots, x_{d-i} \right)
 \geq 0. \]
In particular, ${\ds -\rme_{1}(Q) = \chi_{1}(\bfx)}$  if and only if $\depth(R) \geq d-1$.
 \end{Theorem}

 \begin{proof} We outline the proof for $-\rme_{1}(Q)$. Let 
 \[ \bfx'=x_{1}, \ldots, x_{d-1}, \quad h_{i} = \l( H_{i}(\bfx) ), \quad h_{i}' = \l( H_{i}( \bfx') ), \]
 where $H_{i}(\tratto)$ denotes the homology module of the corresponding Koszul complex.
 By \cite[Theorem 3.6]{Chern7}, we have
 \[ \begin{array}{rcl}
 {\ds  -\rme_{1}(Q) } &=& {\ds  h_{1} - 2 h_{2} + 3 h_{3} + \cdots + (-1)^{j+1} j h_{j} + \cdots + (-1)^{d+1} dh_{d}} \vspace{0.1 in} \\
 &=& {\ds \chi_{1}(\bfx) - \big( h_{2} - 2h_{3} + \cdots + (-1)^{j}(j-1)h_{j} + \cdots + (-1)^{d}(d-1)h_{d}   \big) }
 \end{array} \]
 By using the short exact sequences given in \cite[Corollary 3.2]{Chern7}, we obtain 
 \[ \chi_{1}(\bfx') = h_{2} - 2h_{3} + \cdots + (-1)^{j}(j-1)h_{j} + \cdots + (-1)^{d}(d-1)h_{d}.\]
 This completes the proof. 
 \end{proof}
 
 When the depth of a ring $R$ is sufficiently large, the converse of Theorem~\ref{Ch7-3-7} is also true.

\begin{Theorem}\label{Ch7-4-2}{\rm \cite[Theorem 4.2]{Chern7}}
Let $(R, \fkm)$ be a Noetherian local ring of dimension $d \geq 2$ with $\depth(R) \geq d-1$ and infinite residue field. Let ${\ds Q=(\bfx)=(x_{1}, \ldots, x_{d})}$ be a parameter ideal of $R$. Then the following are equivalent{\rm :}
\begin{enumerate}[{\rm (i)}]
\item ${\ds - \rme_{1}(Q) \leq \chi_{1}(\bfx) }${\rm ;}
\item ${\ds - \rme_{1}(Q) = \chi_{1}(\bfx) }${\rm ;}
\item $\bfx$ is a $d$-sequence{\rm ;}
\item ${\ds \rme_{2}(Q) =0}$.
\end{enumerate}
\end{Theorem}

\begin{proof} We outline the proof ${\ds \mbox{(ii)} \Rightarrow \mbox{(iii)} }$. We may choose a system ${\ds \bfx=x_{1}, \ldots, x_{d}}$ of parameters in $R$ such that $Q=(\bfx)$ and ${\ds \bfx'=x_{1}, \ldots, x_{d-1} }$ is an $R$-regular sequence and a superficial sequence for $R$. Let ${\ds Q'=(\bfx')}$.   Then, by \cite[Proposition 3.3]{Chern7}, we obtain
 \[ \chi_{1}(\bfx) = \l \big( ( 0  :_{R/Q'} x_{d} )  \big). \]
Also, we have
\[ \rme_{1}(Q) = \rme_{1}(Q/Q') = - \l \left( H^{0}_{\fkm} (R/Q')   \right), \quad \mbox{and} \quad  H^{0}_{\fkm} (R/Q')  = \bigcup_{s \geq 1} (0 :_{R/Q'} x_{d}^{s} ). \]
If ${\ds - \rme_{1}(Q) = \chi_{1}(\bfx) }$, then ${\ds ( 0  :_{R/Q'} x_{d} )  = ( 0  :_{R/Q'} x_{d}^{2} ) }$. By \cite[Proposition 2.6]{Chern7},  $\bfx$ is a $d$-sequence.
\end{proof}

Let $(R, \fkm)$ be a Noetherian local ring and $I$ an ideal of $R$. Let $M$ be a finitely generated $R$-module. The {\em associated graded module} of $M$ relative to $I$ is $\G(I, M) = \bigoplus_{n \geq 0} I^{n}M/I^{n+1}M $. Then the {\em $j$-transform} of $M$ relative to $I$ is ${\ds \rmH_{I}(M)= H^{0}_{\fkm}( \G(I, M) ) }$.  The study of the $j$-transform was initiated by Achilles and Manaresi \cite{AM93} who made use of the fact that $\rmH_{I}(M) =  \bigoplus_{k \geq 0} \rmH_{k}$ has an associated numerical function $n \mapsto \psi_{I}^{M}(n)= \sum_{k \leq n} \l(\rmH_{k})$. This function is called the {\em $j$-function} of $M$ relative to $I$ and is a broad generalization of the classical Hilbert function. The {\em $j$-polynomial} of $M$ relative to $I$ is the Hilbert polynomial 
\[ \sum_{i=0}^{\ell} (-1)^{i} j_{i}(I, M){{n+\ell -i}\choose{\ell -i}}, \]
where $\ell$ is the analytic spread of $I$. We call ${\ds  j_{i}(I, M)}$ the {\em $j$-coefficients} of $I$ relative to $M$. In particular, ${\ds j_{0}(I, M)}$ is called the $j$-multiplicity of $I$ relative to $M$.  When ${\ds I=(\bfx)=(x_{1}, \ldots, x_{m})}$, we also write ${\ds j_{i}(I, M) = j_{i}(\bfx, M)}$. In \cite[Corollary 3.3]{X16}, Xie presented a detailed formula for $j_{i}(I)$ for an ideal $I$ with maximal analytic spread and satisfying certain conditions in a Cohen-Macaulay ring. In particular, a formula for $j_{1}(I)$ was used to generalize Northcott's inequality \cite[Theorem 4.1]{X16}.

\begin{Proposition}{\rm \cite[Corollary 2.10]{jmult}}
Let $(R, \fkm)$ be a Noetherian local ring and $M$ a finitely generated $R$-module of positive dimension. Let ${\ds \bfx=x_{1}, \ldots, x_{m}}$ be an amenable partial system of parameters of $M$ that is a $d$-sequence relative to $M$. Suppose that ${\ds \dim_{\G} \left( \rmH_{(\bfx)}(M) \right)=m}$, where $\G=\G((\bfx), R)$ is the associated graded ring. Then ${\ds j_{1}(\bfx, M) \leq 0}$. 
\end{Proposition}

\noindent  Suppose ${\ds \bfx=x_{1}, \ldots, x_{m}}$ is an amenable partial system of parameters of $M$ with  $m < \dim(M)$. A question raised by Vasconcelos and his coauthors is whether the values of ${\ds j_{1}(\bfx, M)}$ would detect various properties of $M$ such as Cohen-Macaulay, Buchsbaum, and generalized Cohen-Macaulay just the way $\rme_{1}(I, M)$ or $\chi_{1}(\bfx, M)$ would do. The significant distinction between ${\ds \G(I, M)}$ and ${\ds \rmH_{I}(M)= H^{0}_{\fkm}( \G(I, M) ) }$ is that the latter may not be homogeneous and therefore the vanishing of some of its Hilbert coefficients does not place them entirely in the context of \cite{Chern2, Chern3, Chern1}.

\begin{Conjecture}{\rm \cite[Conjecture 1.1]{jmult} \label{regular}
Let $(R, \fkm)$ be a Noetherian local ring and $M$ a  finitely generated unmixed $R$-module. 
Let ${\ds \bfx=x_{1}, \ldots, x_{m}}$ be an amenable partial system of parameters of $M$ that is a $d$-sequence relative to $M$. Let $\G=\G((\bfx), R)$ be the associated graded ring. Suppose that ${\ds \dim_{\G} \left(  \rmH_{(\bfx)}(M)  \right) = m}$, and that ${\ds j_{1}(\bfx, M)=0}$. Then $\bfx$ is a regular sequence on $M$.
}\end{Conjecture}

The following gives a partial answer to this conjecture.

\begin{Theorem}{\rm \cite[Theorem 4.10]{jmult}}
Let $(R, \fkm)$ be a Noetherian local ring and $M$ a finitely generated $R$-module of positive dimension. 
Let ${\ds \bfx=x_{1}, \ldots, x_{m}}$ be an amenable partial system of parameters of $M$ that is a $d$-sequence relative to $M$. Suppose that either {\rm (i)} ${\ds m \geq 2}$ and ${\ds \depth(M) \geq m-1}$ or  that {\rm (ii)} $m=3$ and $M$ is unmixed. Then ${\ds j_{1}(\bfx, M)=0}$ if and only if $\bfx$ is a regular sequence on $M$.
\end{Theorem}

A finitely generated module $M$ is said to have {\em finite local cohomology} if ${\ds \l( H^{i}_{\fkm} (M) ) < \infty}$ for all ${\ds i \neq \dim(M)}$. This condition is equivalent to the existence of a standard system of parameters, that is, a system of parameters which is a $d^{+}$-sequence relative to $M$. Moreover, if $M$ has finite local cohomology, then every partial system of parameters of $M$ is amenable \cite{STC78, SV86}. 

\begin{Theorem}{\rm \cite[Theorem 4.13]{jmult}}
Let $(R, \fkm)$ be a Noetherian local ring and $M$ a finitely generated $R$-module of positive depth. Suppose that $M$ has 
finite local cohomology. Let ${\ds \bfx=x_{1}, \ldots, x_{m}}$ be an amenable partial system of parameters of $M$ that is a $d$-sequence relative to $M$. Then ${\ds j_{1}(\bfx, M)=0}$ if and only if $\bfx$ is a regular sequence on $M$.
\end{Theorem}

The following theorem extends Theorems \ref{Ch7-3-7} and \ref{Ch7-4-2} to partial amenable systems of parameters.

\begin{Proposition}{\rm \cite[Corollary 4.15 and Proposition 4.16]{jmult}}
Let ${\ds \bfx=x_{1}, \ldots, x_{m}}$ be an amenable partial system of parameters of $M$ that is a $d$-sequence relative to $M$.
\begin{enumerate}[{\rm (1)}]
\item For all $0 \leq i \leq m$, we have
\[ (-1)^{i} j_{i}(\bfx, M) = \chi_{1} \left(x_{1}, \ldots, x_{m-i}, x_{m-i+1}, M \right) - \chi_{1} \left(x_{1}, \ldots, x_{m-i} , M\right). \]

\item Suppose $m \geq 2$. Then ${\ds - j_{1}(\bfx, M) \leq \chi_{1}(\bfx, M) }$, 
where the equality holds true if and only if $x_{1}, \ldots, x_{m-1}$ is a regular sequence on $M$.
\end{enumerate}
\end{Proposition}

Recall that a module $M$ is generalized Cohen-Macaulay if ${\ds H^{i}_{\fkm}(M)}$ is finitely generated for all $i \neq \dim(M)$. Theorem \ref{GO11-1-1} shows the relationship between the finiteness of the set of Hilbert coefficients and generalized-Cohen-Macaulayness of the ring. The following theorem shows a similar result in terms of $j$-coefficients.

\begin{Theorem}{\rm \cite[Theorem 5.1]{jmult}}
Suppose that there exists a system  $\bfx$ of parameters of $M$ which is a strong $d$-sequence relative to $M$. Let $m$ be an integer such that ${\ds 0 < m < \dim(M)}$.  Let  ${\ds \Omega(M) = \{ j_{i}(\bfx, M) \mid 0 \leq i \leq m-1 \} }$, where ${\ds \bfx=x_{1}, \ldots, x_{m}}$ is an amenable partial system of parameters of $M$ which is a $d$-sequence relative to $M$. 
Then the set $\Omega(M)$ is finite if and only if $H^{i}_{\fkm}(M)$ is a finitely generated $R$-module for every $1 \leq i \leq m$.
\end{Theorem}

\bigskip

\section{Variation of Hilbert Coefficients}\label{Variations}

 An interesting question to study, referred to as variation of Hilbert coefficients, is how the Hilbert coefficients $\rme_{i}(J)$ change when an $\fkm$-primary ideal $J$ is enlarged. First we recall the definition of a special fiber.

\begin{Definition}\label{spfiber}{\rm
Let $(R, \fkm)$ be a Noetherian local ring and let $I$ be an ideal in $R$. 
The {\em special fiber ring} $\clF(I)$ of $I$ is defined as 
\[ \clF(I) = \clR(I)/\fkm\clR(I) =   \bigoplus_{n \geq 0} I^{n}/ \fkm I^{n}. \] 
The dimension of $\clF(I)$ is called the {\em analytic spread} of $I$ and is denoted by $\ell(I)$. 
}\end{Definition}

The Hilbert function of $\clF(I)$ measures the growth of the minimal number of generators $\nu(I^{n})$ of the powers of $I^{n}$.  The Hilbert polynomial of $\clF(I)$ has degree $\ell(I)-1$ and its leading coefficient is denoted by $f_{0}(I)$, which can be used to bound a variation of Hilbert coefficients as shown below.

\begin{Theorem}{\rm \cite[Theorem 2.2]{Chern4} }
Let $(R, \fkm)$ be a Noetherian local ring and let  $J$ and ${\ds I=(J, h)}$ be $\fkm$-primary ideals. Then
\[ \rme_{0}(J) - \rme_{0}(I) \leq \l(R/ (J:I) ) \cdot f_{0}(J), \]
where ${\ds f_{0}(J)}$ is the multiplicity of the special fiber of the Rees algebra of $J$.
\end{Theorem}

When multiple additional generators are added to $J$, a similar result holds.

\begin{Theorem}\label{Ch4-2-6}{\rm \cite[Theorem 2.6]{Chern4} }
Let $(R, \fkm)$ be a Noetherian local ring of positive dimension. Let $J$ be an $\fkm$-primary ideal and let $I=(J, h_{1}, \ldots, h_{m} )$ be integral over $J$ of reduction number $r$. Then 
\[ \rme_{1}(I) - \rme_{1}(J) \leq  \l( R/(J:I) ) \cdot \left[ {{m+r}\choose{r}} -1  \right] \cdot f_{0}(J), \]
where ${\ds f_{0}(J)}$ is the multiplicity of the special fiber of the Rees algebra of $J$.
\end{Theorem}

Variation of the Chern coefficients is related to the multiplicity of certain Sally modules. Let $Q$ be a minimal reduction of an $\fkm$-primary ideal $I$. Suppose that the dimension of the Sally module $S_{Q}(I)$ of $I$ relative to $Q$ is equal to the dimension $d$ of $R$ and that ${\ds H^{0}_{\fkm}(R) \subset I}$. Then the multiplicity $s_{0}(Q, I)$ of the Sally module  $S_{Q}(I)$ is 
\[ s_{0}(Q, I) = \rme_{1}(I) - \rme_{1}(Q) - \rme_{0}(I) + \l(R/I). \]
Therefore as a consequence of Theorem~\ref{Ch4-2-6}, we obtain the following.
\[ s_{0}(Q, I) \leq - \rme_{0}(I) + \l(R/I) + \l( R/(Q:I) ) \cdot \left[ {{\nu(I)-d+r}\choose{r}} -1  \right],\]
where $r$ is the reduction number of $I$ relative to $Q$ \cite[Corollary 2.9]{Chern4}. Notice that the reduction number plays a key role in the upper bound for the variation of Chern coefficients and the multiplicity of Sally module. Thus it is useful to find more information about the reduction number in terms of the Hilbert coefficients. Let $(R, \fkm)$ be a Cohen-Macaulay ring of dimension $d \geq 1$ with infinite residue field.  Vasconcelos proved that 
\[ r(I) \leq \frac{d \cdot \rme_{0}(I)}{o(I)} -2d +1, \] where $o(I)$ is the $\fkm$-adic order of $I$ \cite[Theorem 2.45]{V05Book}. Another distinctive bound was introduced by Rossi \cite[Corollary 1.5]{R00}. If $d \leq 2$, then  
\[ r(I) \leq  \rme_{1}(I) - \rme_{0}(I) + \l(R/I) +1. \]
There is a special case where the equality in \cite[Corollary 1.5]{R00} holds true in a Cohen-Macaulay ring of any positive dimension.

\begin{Proposition}{\rm \cite[Proposition 2.3]{Red}}
Let $(R, \fkm)$ be a Cohen-Macaulay local ring of positive dimension and let $I$ be an $\fkm$-primary ideal. Suppose that there exists a minimal reduction $Q$ of $I$ such that ${\ds \l(I^{2}/QI) =1}$. Then the minimal reduction number of $I$ relative to a minimal reduction $J$ is independent of the choice of $J$ and 
\[ r(I) =  \rme_{1}(I) - \rme_{0}(I) + \l(R/I) +1. \]
\end{Proposition}

It is still an open question whether the result  \cite[Corollary 1.5]{R00} of Rossi can be extended to a Cohen-Macaulay ring of higher dimension or a general Noetherian ring. Ghezzi, Goto, Hong, Vasconcelos were able to extend Rossi's inequality to Buchsbaum and generalized Cohen-Macaulay rings of dimension $2$.

\begin{Theorem}{\rm \cite[Theorem 4.3]{Red}} \label{min red}
Let $(R, \fkm)$ be a Buchsbaum local ring of $\dim(R)=2$ and $\depth(R) >0$. Let $S$ be the $S_{2}$-fication of $R$.  Let $I$ be an $\fkm$-primary ideal with a minimal reduction $Q$. 
\begin{enumerate}[{\rm (1)}]
\item If  ${\ds I=IS}$, then  ${\ds r_{Q}(I) \leq 1+ \rme_{1}(I) - \rme_{0}(I) + \l(R/I)  - \rme_{1}(Q)}$. 
\item If ${\ds I \subsetneq IS}$, then ${\ds r_{Q}(I) \leq ( 1- \rme_{1}(Q))^{2} s_{0}(Q, I) - 2 \rme_{1}(Q) +1}$. 
\end{enumerate}
\end{Theorem}

\begin{proof} We outline the proof. Because $R$ is Buchsbaum, we have ${\ds -\rme_{1}(Q) = \l (H^{1}_{\fkm}(R) ) = \l(S/R) }$. Let ${ \left\{ \fkm_{1}, \ldots, \fkm_{t} \right\} }$ be the set of maximal ideals of $S$. For each $i=1, \ldots, t$, set ${\ds S_{i} = S_{\fkm_{i}}}$ and let $f_{i} = [S/\fkm_{i} : R/\fkm]$ denote the relative degree. We may assume that ${\ds r_{QS}(IS) = r_{QS_{1}}(IS_{1})}$. 

\medskip

\noindent Suppose that $I=IS$.  By using \cite[Corollary 1.5]{R00}, we obtain the following.
\[ \begin{array}{rcl}
{\ds  r_{Q}(I) = r_{QS}(I) = r_{QS}(IS) =  r_{QS_{1}}(IS_{1}) } &\leq & {\ds 1 + \rme_{1}(IS_{1}) - \rme_{0}(IS_{1}) + \l( S_{1}/IS_{1}) } \vspace{0.1 in} \\ & \leq & {\ds 1 + \sum_{i=1}^{t} \Big(  \rme_{1}(IS_{i}) - \rme_{0}(IS_{i}) + \l( S_{i}/IS_{i})   \Big) f_{i}}  \vspace{0.1 in} \\
&= & {\ds 1+ \rme_{1}(IS) - \rme_{0}(IS) + \l(S/IS)  } \vspace{0.1 in} \\
&= & {\ds 1+ \rme_{1}(IS) - \rme_{0}(IS) +  \l(R/IS) - \rme_{1}(Q)  } \vspace{0.1 in} \\
&= & {\ds 1+ \rme_{1}(I) - \rme_{0}(I) +  \l(R/I) - \rme_{1}(Q)  } 
\end{array} \]  

\medskip

\noindent Suppose that $I \subsetneq IS$. By using \cite[Proposition 4.1, Theorem 2.6]{Red} and \cite[Proposition 1.118]{V05Book}, we obtain the following. 
\[ \begin{array}{rcl}
{\ds r_{Q}(I) } & \leq & {\ds \big(r_{QS}(IS) +1\big) \nu(S) - 1 } \vspace{0.1 in} \\
& \leq & {\ds \big( s_{0}(QS, IS) +2 \big) \nu(S) - 1  } \vspace{0.1 in} \\
& \leq & {\ds \big( s_{0}(Q, I) \nu(S) + 2 \big) \nu(S) -1 } \vspace{0.1 in} \\
&=& {\ds \big( s_{0}(Q, I) (1- \rme_{1}(Q)) + 2\big) \big(1- \rme_{1}(Q)\big)  -1 }  \vspace{0.1 in} \\
&=& {\ds ( 1- \rme_{1}(Q))^{2} s_{0}(Q, I) - 2 \rme_{1}(Q) +1 }
\end{array}\]
\end{proof}

\begin{Corollary}{\rm \cite[Corollary 4.6]{Red}} 
Let $(R, \fkm)$ be a generalized Cohen-Macaulay local ring of $\dim(R)=2$ and $\depth(R) >0$. Let $S$ be the $S_{2}$-fication of $R$.  Let $I$ be an $\fkm$-primary ideal with a minimal reduction $Q$. Suppose that ${\ds I=IS}$. Then
\[ r_{Q}(I) \leq \rme_{1}(I) - \rme_{0}(I) + \l(R/I) + 1 + \l(H^{1}_{\fkm}(R)). \]
\end{Corollary}

Ghezzi, Goto, Hong, Vasconcelos introduced some uniform bounds for reduction numbers.

\begin{Proposition}{\rm \cite[Proposition 3.1]{Red}}
Let $(R, \fkm)$ be a Noetherian local ring with infinite residue field. Let $I$ be an $\fkm$-primary almost complete intersection ideal with a minimal reduction $Q$. If ${\ds \l(I^{n}/QI^{n-1}) = 1}$ for some $n$, then ${\ds r(I) \leq n \nu(\fkm) -1}$.
\end{Proposition}

\bigskip

\section{Sally Modules and Special Fibers}\label{Sally}

Let $Q=(a_{1}, \ldots, a_{d})$ be a minimal reduction of $I$. Suppose that the Sally module $S_{Q}(I)$ has dimension $d$. Then the elements $a_{1}t, \ldots, a_{d}t$ form a system of parameters for $S_{Q}(I)$. We denote the corresponding multiplicity  by ${\ds s_{0}(Q, I)=\rme_{0}((Qt), S_{Q}(I)) }$.

\begin{Proposition}{\rm (\cite[Corollary 3.3]{V94}, \cite[Proposition 2.8]{C09})}\label{Sally mult prop}
Let $(R, \fkm)$ be a Noetherian local ring of dimension $d$ with infinite residue field. Let $I$ be an $\fkm$-primary ideal with a minimal reduction $Q$. 
Suppose that $S_{Q}(I)$ has dimension $d$.
\begin{enumerate}[{\rm (1)}]
\item Suppose $R$ is Cohen-Macaulay. Then ${\ds s_{0}(Q, I) = \rme_{1}(I) -\rme_{0}(I)+ \l(R/I) }$. In particular, ${\ds s_{0}(Q, I)}$ is independent of the choice of $Q$. 
\item In general, ${\ds s_{0}(Q, I) \leq \rme_{1}(I) -\rme_{0}(I) - \rme_{1}(Q)+ \l(R/I) }$. Morevoer, equality holds if and only if $I$ contains ${\ds (x_{1}, \ldots, x_{d-1}): \fkm^{\infty} }$, where ${\ds x_{1}, \ldots, x_{d-1}}$ are general elements of $Q$.
\end{enumerate}
\end{Proposition}

Let $I$ be an $\fkm$-primary ideal with a reduction $J$.  There is another approach to the Sally module $S_{J}(I)$ by using the following exact sequence:
\[ \clR(J) \oplus \clR(J)^{m-d}[-1] \lar \clR(I) \lar S_{J}(I)[-1] \lar 0, \]
where $m=\nu(I)$. By tensoring this exact sequence with $R/\fkm$, we obtain the following exact sequence:
\[ \clF(J) \oplus \clF(J)^{m-d}[-1] \lar \clF(I) \lar S_{J}(I)[-1] \otimes R/\fkm \lar 0. \]
We denote the multiplicity of the special fiber ring $\clF(I)$ by $f_{0}(I)$. Recall that the Hilbert function of $\clF(I)$ measures the growth of the minimal number of generators ${\ds \nu(I^{n}) }$ of the powers of $I$. 
When using the various graded algebras, it is useful to understand how the multiplicities interact.  One nice example of such information is provided by the following theorem.

\begin{Theorem}{\rm \cite[Proposition 2.4]{V03-1}}
Let $(R, \fkm)$ be a Cohen-Macaulay local ring and let $I$ be an $\fkm$-primary ideal. Then
\[ f_0(I) \leq \inf \{\rme_{0}(I),\;  \rme_{1}(I) +1 \}.\]
\end{Theorem}

This upper bound for $f_{0}(I)$ was modified first in a Cohen-Macaulay ring, and then in a Noetherian ring.

\begin{Proposition}\label{CPV06-2-2}{\rm  (\cite[Theorem 2.1, Proposition 2.2]{CPV06}, \cite[Theorem 3.4]{C09})} \label{fiber cone mult}
Let $(R, \fkm)$ be a Noetherian local ring of dimension $d>0$ with infinite residue field.  Let $I$ be an $\fkm$-primary ideal. 
\begin{enumerate}[{\rm (1)}]
\item Suppose $R$ is Cohen-Macaulay. Then
\[ f_{0}(I) \leq \rme_{1}(I) - \rme_{0}(I) + \l(R/I) + \nu(I) - d+1 \leq \rme_{1}(I) +1.\]
In particular, if ${\ds f_{0}(I) = \rme_{1}(I) +1}$, then $\fkm I = \fkm Q$ for any minimal reduction $Q$ of $I$.

\item In general, we have
\[ f_{0}(I) \leq \rme_{1}(I) - \rme_{0}(I)  - \rme_{1}(Q) + \l(R/I) + \nu(I) - d+1,\]
where $Q$ is any minimal reduction of $I$.
\end{enumerate}
\end{Proposition}

We denote the associated graded ring of $I$ by $\G(I)$.

\begin{Proposition}{\rm  \cite[Theorem 2.5, Corollaries 2.6 and 3.1]{CPV06}.}
Let $(R, \fkm)$ be a Cohen-Macaulay local ring of dimension $d>0$ with infinite residue field and $I$ an $\fkm$-primary ideal.  Suppose that ${\ds f_{0}(I) = \rme_{1}(I) - \rme_{0}(I) + \l(R/I) + \nu(I) - d+1}$.
\begin{enumerate}[{\rm (1)}]
\item The special fiber $\clF(I)$ is unmixed.
\item ${\ds \depth( \clF(I) ) \geq \min \{ \depth \G(I) +1, \; d \} }$. Moreover, if ${\ds \depth ( \G(I) ) \geq d-1}$, then $\clF(I)$ is Cohen-Macaulay.
\item Suppose the residue field $R/\fkm$ has characteristic $0$. Then ${\ds r(I) \leq f_{0}(I) +1}$. 
\end{enumerate} 
\end{Proposition}

\begin{Corollary}\label{CPV06-3-5}{\rm \cite[Remark 3.2, Corollary 3.5]{CPV06} }
Let $(R, \fkm)$ be a Cohen-Macaulay local ring of dimension $d>0$ with infinite residue field. Let $I$ be an $\fkm$-primary ideal.  Suppose that $\clF(I)$ is Cohen-Macaulay. Then
\[ r(I) \leq f_{0}(I) - \nu(I) +d \leq \rme_{1}(I) - \rme_{0}(I) + \l(R/I) +1. \]
\end{Corollary}

The following theorem shows that for a very special class of ideals we obtain not only the Cohen-Macaulayness of the special fiber, but also the equalities in Corollary~\ref{CPV06-3-5}.

\begin{Theorem}{\rm \cite[Theorem 3.7]{Red} }
Let $(R, \fkm)$ be a Cohen-Macaulay local ring of positive dimension with infinite residue field. Let $I$ be an $\fkm$-primary ideal which is an almost complete intersection. Suppose that $\l(I^{2}/QI)=1$ for some minimal reduction $Q$ of $I$. Then the special fiber ring $\clF(I)$ is Cohen-Macaulay. In particular, we have
\[ r(I) = f_{0}(I) -1 = \rme_{1}(I) - \rme_{0}(I) + \l(R/I) +1. \]
\end{Theorem}

\bigskip

\section{Epilogue}\label{epilogue}

One of the projects that Vasconcelos was most enthusiastic about was to create a universal degree function $\Deg$ which assigns to an algebraic structure $\clA$ the value $\Deg(\clA)$. There are two grand objectives he wanted to achieve from this general endeavor. One is to obtain an enhanced version of ordinary multiplicity equipped with suitable properties regarding hyperplane sections and short exact sequences.  The other is that $\Deg$ would be instrumental in finding estimates for the complexity of normality or determining the number of generators  in the most general settings just as the ordinary multiplicities do under a restricted setting. Vasconcelos' lifelong endeavor to achieve these goals resulted in his numerous publications on arithmetic degree, geometric degree, jdeg, homological degree, cohomological degree, bdeg, canonical degree, and bi-canonical degree. The survey paper \cite{GH} summarizes and highlights Vasconcelos' outstanding contributions to the theory of aforementioned degrees.

\end{document}